\documentclass[preprint, authoryear]{elsarticle}
\journal{Journal of Mathematical Analysis and Applications}

\usepackage{lineno,hyperref}
\usepackage{amsmath, amsthm, amssymb}
\usepackage{mathtools}
\usepackage{bbm}
\usepackage{csquotes}
\usepackage{pgfplots}
\usepackage{pgfplotsthemetol}
\pgfplotsset{compat=1.11}
\usepgfplotslibrary{fillbetween}
\usepgfplotslibrary{groupplots}
\usepackage{tikz}
\usepackage{tikz-cd}
\usetikzlibrary{arrows, positioning}
\usepackage{subcaption}
\usepackage{enumitem}
\usepackage{bm}

\newtheorem{theorem}{Theorem}[section]
\newtheorem{definition}[theorem]{Definition}
\newtheorem{lemma}[theorem]{Lemma}
\newtheorem{corollary}[theorem]{Corollary}
\newtheorem{proposition}[theorem]{Proposition}
\newtheorem{remark}[theorem]{Remark}
\newtheorem*{example}{Example}

\biboptions{authoryear}

\usepackage[margin=3cm]{geometry}
\providecommand{\R}{\ensuremath{\mathbb{R}}}
\providecommand{\C}{\ensuremath{\mathcal{C}}}	
\providecommand{\Cd}[1]{\C_{#1}}				
\providecommand{\M}{\ensuremath{\mathcal{M}}}	
\providecommand{\Md}[1]{\M_{#1}}				
\providecommand{\SI}[1]{\mathcal{S}^{#1}}		

\providecommand{\boldm}[1]{\bm{#1}}						
\providecommand{\rbraces}[1]{\left( #1 \right)} 		
\providecommand{\cbraces}[1]{\left[ #1 \right]}			
\providecommand{\braces}[1]{\left\{ #1 \right\}} 		
\providecommand{\abs}[1]{\left\lvert #1 \right\rvert} 	
\providecommand{\norm}[1]{\left\lVert #1 \right\rVert} 	

\newcommand{\set}[2]{\left\{ #1 \; \left \arrowvert \; #2 \right. \right\}} 
\providecommand{\indFunc}[1]{\mathbbm{1}_{#1}} 
\providecommand{\lebesgue}[3]{\int\limits_{#2} #1 \ \mathrm{d}#3}		
\providecommand{\rInt}[4]{\int\limits_{#2}^{#3} #1 \ #4}				
\providecommand{\cInt}[4]{\rInt{#1}{#2}{#3}{\mathrm{d}#4}}				

\providecommand{\TDLn}{\Lambda}									
\providecommand{\TD}[2]{\Lambda \rbraces{#1 \; ; #2}}		
\providecommand{\TDL}[2]{\TD{#1}{#2}}						
\providecommand{\TDLphFo}[1]{\widetilde{#1}}				
\providecommand{\TDLph}[1]{\widetilde{#1}}					
\providecommand{\TDLprod}{\TDLn^{0}}						

\providecommand{\Mgprod}[3]{\phi_{#3, #2} \rbraces{#1}}
\providecommand{\Mprod}[2]{\phi_{#2} \rbraces{#1}}

\begin{document}

\begin{frontmatter}

\title{A Markov product for tail dependence functions\tnoteref{t1,t2}}
\tnotetext[t1]{\textcopyright 2021. This manuscript version is made available under the CC-BY-NC-ND 4.0 license \url{http://creativecommons.org/licenses/by-nc-nd/4.0/}.}
\tnotetext[t2]{Accepted for publication in the Journal of Mathematical Analysis and Applications (\hyperlink{https://doi.org/10.1016/j.jmaa.2021.124942}{10.1016/j.jmaa.2021.124942}).}
\author[tud]{Karl Friedrich Siburg}\corref{correspondingauthor}%
\cortext[correspondingauthor]{Corresponding author.}%
\author[tud]{Christopher Strothmann\fnref{scholarship}}%
\fntext[scholarship]{Supported by the German Academic Scholarship Foundation.}%
\address[tud]{Faculty of Mathematics, TU Dortmund University, Germany}%

\begin{abstract}
We introduce a Markov product structure for multivariate tail dependence functions, building upon the well-known Markov product for copulas. 
We investigate algebraic and monotonicity properties of this new product as well as its role in describing the tail behaviour of the Markov product of copulas.
For the bivariate case, we show additional smoothing properties and derive a characterization of idempotents together with the limiting behaviour of $n$-fold iterations. 
Finally, we establish a one-to-one correspondence between bivariate tail dependence functions and a class of positive, substochastic operators.
These operators are contractions both on $L^1(\R_+)$ and $L^\infty(\R_+)$ and constitute a natural generalization of Markov operators. 
\end{abstract}

\end{frontmatter}

\section{Introduction}

In many applications, there is a need to quantify the dependence between different random variables. 
Examples range from finance to hydrology, where the dependence can have a global, for instance a linear or a functional, or a local character. 
In the following, we are interested in a certain type of local dependence, the tail dependence, which describes the extremal behaviour between multiple random variables. 
A natural application are the joint losses of multiple stocks in a portfolio.
The lower tail dependence function 
\begin{equation*}
	\Lambda((w_1, w_2) ; X, Y) := w^{-1}_1 \lim\limits_{s \searrow 0} \mathbb{P} \rbraces{X \leq F_X^{-1}(sw_1) \; | \; Y \leq F_Y^{-1}(sw_2)} = \lim\limits_{s \searrow 0} \frac{C_{XY}(sw_1, sw_2)}{s}
\end{equation*}
of two continuous random variables $X$ and $Y$ allows a scale-free characterization of the joint behaviour in the extremes, in this case the jointly occurring extreme losses.
Properties and applications of the tail dependence functions can be found in \cite{Joe.2014}, while estimators and their statistical properties have been established in \cite{Schmidt.2006}.

This paper treats the tail properties of a certain class of $d$-variate copulas, namely the ones constructed via the (generalized) Markov product. 
For the set of $2$-copulas, denoted by $\Cd{2}$, the Markov product $*$ has become an important tool in the modelling and description of dependencies. 
First introduced by \cite{Darsow.1992} to model transition probabilities in the context of Markov processes through a rephrasing of the Chapman-Kolmogorov-equations in terms of consistency conditions imposed on a family of copulas, it also plays an essential role in the theory of complete dependence (see \cite{Siburg.2008} and \cite{Trutschnig.2011}) and the study of extremal elements (see \cite{Darsow.1992}). 
An extensive overview over the properties and applications of the Markov product can be found in \cite{Durante.2015}.
Some results of the tail behaviour of similar constructions have been achieved in the context of vine-copulas by \cite{Joe.2010} and more recently by \cite{Jaworski.2015}.
To facilitate the study of the extremal behaviour of the Markov product on $\Cd{2}$, we introduce a generalized version of the Markov product on the set of all tail dependence functions $\Md{2}$ and link both under appropriate regularity conditions.  
One of the most important properties of the Markov product $\Md{2}$ is a monotonicity property, which results in an overall dependence reduction and, in general, does not hold for $2$-copulas. 
Using this monotonicity, we treat iterates of the Markov product and derive additional smoothing properties akin to those presented in \cite{Trutschnig.2013b}.
Finally, we connect the set of all bivariate tail dependence functions equipped with $*$ to a certain class of substochastic operators and their composition which generalize the well-known Markov operators. 

The paper is structured as follows.  
Section \ref{section:basics} presents the necessary notation and some preliminaries.
Section \ref{section:product} introduces the Markov product for tail dependence functions and establishes a link to the Markov product of copulas. 
Section \ref{section:monotonicity} discusses the monotonicity of the Markov product, while Section \ref{section:iterates} employs these results to derive the behaviour of iterates, idempotents and averages. 
Finally, Section \ref{section:operator} links $(\mathcal{M}, *)$ to a class of linear operators $(T, \circ)$.

\section{Notation and preliminaries} \label{section:basics}

A $d$-copula is a $d$-variate distribution function on $[0, 1]^d$ with uniform univariate margins. 
We write $\R_+ := [0, \infty)$ and use bold symbols to denote vectors, e.g. $\boldm{w} = (w_1, \ldots, w_d) \in \R_+^d$.

\begin{definition}  \label{def:tail_dependence_function}
For a $d$-copula $C$, the lower tail dependence function $\TDL{\cdot}{C}: \R_+^d \rightarrow \R_+$ is defined as
\begin{gather*}
		\TDLn(\boldm{w}) := \TDL{\boldm{w}}{C} := \lim\limits_{s \searrow 0} \frac{C(s \boldm{w})}{s} ~,
\end{gather*}
provided that the limit exists for all $\boldm{w}$ in $\R_+^d$. 
\end{definition}

Let $\Cd{d}$ and $\Md{d}$ denote the set of $d$-copulas and the set of $d$-variate tail dependence functions, respectively.
We refer to the lower Fréchet-Hoeffding-bound by $C^-$, the upper Fréchet-Hoeffding-bound by $C^+$ and the product copula by $\Pi$. 
Many properties of the copula $C$ immediately transfer to the tail dependence function $\TDLn$ (see, Propositions~4 and 6 in \cite{Jaworski.2006}):

\begin{proposition} \label{prop:tail_dep_func}
A function $\TDLn: \R_+^d \rightarrow \R_+$ is the tail dependence function of a copula $C$ if and only if 
\begin{enumerate}
	\item $\TDLn$ is bounded from below by $0$ and from above by $\TDLn^+ := \TDL{\cdot}{C^+}$.\label{prop:tail_dep_func_bounded}
	\item $\TDLn$ is $d$-increasing, i.e.\ the $\Lambda$-volume of any rectangle in $\R_+^d$ is nonnegative.
	\item $\TDLn$ is homogeneous of order $1$, i.e. $\TDLn(s\boldm{w}) = s \TDLn(\boldm{w})$ for any $s \in \R_+$ and $\boldm{w} \in \R^d_+$.\label{prop:tail_dep_func_homogeneous}
\end{enumerate}  
Furthermore, for any tail dependence function $\TDLn$, we have 
\begin{enumerate}[label=\alph*.]
  \item $\TDLn$ is Lipschitz continuous with constant $1$. 
  \item $\TDLn$ is concave. 
  \item $w_1 \mapsto \partial_1 \TDLn(w_1, w_2, \ldots, w_d)$ is decreasing for  almost all $w_1 \in \R_+$ and all $w_2, \ldots, w_d \in \R_+$.
\end{enumerate}
Finally, for $d=2$, $w_2 \mapsto \partial_1 \TDLn(w_1, w_2)$ is increasing for almost all $w_1 \in \R_+$ and all $w_2 \in \R_+$.
\end{proposition}

The partial derivatives of $d$-copulas and $d$-variate tail dependence functions are only defined almost everywhere. 
We will suppress this fact in the rest of this article.

\begin{remark} \label{remark:bivariate_homogeneity}
Due to the positive homogeneity of $\TDLn$, in the bivariate case, we will often only consider $\TDLn$ on the unit simplex $\SI{1} := \set{\boldm{w} \in \R_+^2}{\boldm{w} = (t, 1-t) \text{ with } t \geq 0 }$ and identify $\SI{1}$ with $[0, 1]$ using 
\begin{equation}
	\TDLphFo{\Lambda}(t) := \TDLn(t, 1-t) ~.
\end{equation} 
\end{remark}

\begin{remark}
Following \cite{Jaworski.2006}, the tail dependence function $\TDLn_C$ of a copula $C$ is the first order approximation of $C$ in zero, that is
\begin{equation*}
	C(\boldm{u})	= \TDLn_C(\boldm{u}) + o(\norm{\boldm{u}})	\quad \text{ whenever } \norm{\boldm{u}} \rightarrow 0 ~. 
\end{equation*}
As $\TDLn_C$ is a local approximation of $C$ in zero, the usual $d_\infty$-metric is unable to differentiate between the tail behaviour of copulas.
Let us illustrate this point using the patchwork technique described in \cite{Durante.2013}.
For any given $d$-variate tail dependence function $\Lambda$, there exists a family of copulas with tail dependence function $\Lambda$, which is dense in the set of all $d$-copulas w.r.t the uniform topology.
This also implies that the class of all tail dependent $d$-copulas is dense in $\Cd{d}$. 
On the other hand, the class of $d$-copulas which do not allow for a tail dependence function is also dense in $\Cd{d}$ w.r.t the uniform topology.
\end{remark}

We investigate a generalized version of the Markov product introduced by \cite{Darsow.1992}, which was discussed in \cite{Jaworski.2015} in the context of vine-copulas. 

\begin{definition}
Let $C_1, \ldots, C_d$ be $2$-copulas and let $C$ be a $d$-copula. 
Then, the $(d+1)$-copula
\begin{equation*}
	\Mgprod{C_1, \ldots, C_d}{C}{u}(v_1, \ldots, v_d) := \cInt{C \rbraces{\partial_1 C_1(t, v_1), \ldots, \partial_1 C_d(t, v_d)}}{0}{u}{t} 
\end{equation*}
is called the $C$-lifting of the copulas $C_1, \ldots, C_d$.
Furthermore, we define the $d$-copula 
\begin{align*}
	\Mprod{C_1, \ldots, C_d}{C}(v_1, \ldots, v_d) 	&:= \cInt{C \rbraces{\partial_1 C_1(t, v_1), \ldots, \partial_1 C_d(t, v_d)}}{0}{1}{t} \\
													&= \Mgprod{C_1, \ldots, C_d}{C}{1}(v_1, \ldots, v_d)  
\end{align*}
to be the generalized Markov product of $C_1, \ldots, C_d$ induced by $C$. 
\end{definition}

Note that for $d=2$, the previously defined generalized Markov product
\begin{align*}
	\Mprod{C_1, C_2}{C}(v_1, v_2) 	= \cInt{C \rbraces{\partial_1 C_1(t, v_1), \partial_1 C_2(t, v_2)}}{0}{1}{t} 
\end{align*}
maps $\Cd{2} \times \Cd{2}$ onto $\Cd{2}$ and is closely related to the traditional Markov product of $2$-copulas via 
\begin{equation*}
	C_1 * C_2 (v_1, v_2) = \Mprod{C_1^T, C_2}{\Pi}(v_1, v_2) ~,
\end{equation*}
where $C_1^T (v_1, v_2) := C_1(v_2, v_1)$. 
\section{A Markov product for tail dependence functions} \label{section:product}

Similar to this construction of higher dimensional copulas from bivariate copulas, we introduce an operation for bivariate tail dependence functions. 

\begin{definition}
Let $\Lambda_1, \ldots, \Lambda_d \in \Md{2}$ and $C \in \Cd{d}$. 
We call 
\begin{equation*}
	\Mgprod{\Lambda_1, \ldots, \Lambda_d}{C}{w_0}(w_1, \ldots, w_d) := \cInt{C \rbraces{\partial_1 \Lambda_1 \rbraces{t, w_1}, \ldots, \partial_1 \Lambda_d \rbraces{t, w_d}}}{0}{w_0}{t}
\end{equation*}
the $C$-lifting of the tail dependence functions $\Lambda_1, \ldots, \Lambda_d$. 
Similarly, the generalized Markov product of $\Lambda_1, \ldots, \Lambda_d$ induced by $C$ is defined by
\begin{equation*}
	\Mprod{\Lambda_1, \ldots, \Lambda_d}{C}(w_1, \ldots, w_d) := \cInt{C \rbraces{\partial_1 \Lambda_1 \rbraces{(t, w_1)}, \ldots, \partial_1 \Lambda_d \rbraces{(t, w_d)}}}{0}{\infty}{t} ~.
\end{equation*}
\end{definition}

First, we verify that the $C$-lifting and the generalized Markov product do in fact generate new tail dependence functions. 

\begin{theorem}
Suppose $C \in \Cd{d}$ and $\Lambda_1, \ldots, \Lambda_d \in \Md{2}$. 
Then $\Mgprod{\Lambda_1, \ldots, \Lambda_d}{C}{w_0}$ and $\Mprod{\Lambda_1, \ldots, \Lambda_d}{C}$ are $(d+1)$-variate and $d$-variate tail dependence functions, respectively. 
\end{theorem}

\begin{proof} 
First, the tail dependence functions $\Lambda_\ell$, for $\ell = 1, \ldots, d$, are positive, monotone in each component, Lipschitz continuous and thus have partial derivatives almost everywhere. 
Moreover, the partial derivatives attain values in $[0, 1]$. 
Therefore, we have
\begin{align*}
	\Mgprod{\Lambda_1, \ldots, \Lambda_d}{C}{w_0}(w_1, \ldots, w_d)
		&\leq \cInt{C^+ \rbraces{\partial_1 \Lambda_1(t, w_1), \ldots, \partial_1 \Lambda_d(t, w_d)}}{0}{\infty}{t} \\
		&\leq \min\limits_{\ell = 1, \ldots, d} \norm{\partial_1 \Lambda_\ell(t, w_\ell)}_1 
		= \min\limits_{\ell = 1, \ldots, d} \rbraces{\lim\limits_{t \rightarrow \infty} \Lambda_\ell(t, w_\ell)} \\
		&\leq \min\limits_{\ell = 1, \ldots, d} w_\ell < \infty ~,
\end{align*}
which establishes the existence of the integral. 
The second inequality is due to $\TDLn$ being increasing in each component and bounded above by $\TDLn^+$. 
Thus, we can define
\begin{equation*}
	\phi(\boldm{w}) := \cInt{C \rbraces{\partial_1 \Lambda_1(t, w_1), \ldots, \partial_1 \Lambda_d(t, w_d)}}{0}{w_0}{t} ~.
\end{equation*}
It remains to verify properties \ref{prop:tail_dep_func_bounded}. - \ref{prop:tail_dep_func_homogeneous}. of Proposition~\ref{prop:tail_dep_func}, which characterizes tail dependence functions.
For the first property, note that due to all copulas being bounded from above by $C^+$ and as tail dependence functions have bounded partial derivatives between $0$ and $1$, it holds
\begin{align*}
		0 	&\leq \cInt{C \rbraces{\partial_1 \Lambda_1(t, w_1), \ldots, \partial_1 \Lambda(t, w_d)}}{0}{w_0}{t} \\
			&\leq \begin{cases}
				\cInt{1}{0}{w_0}{t} = w_0 \\
				\cInt{C^+ \rbraces{\partial_1 \Lambda_1(t, w_1), \ldots, \partial_1 \Lambda_d(t, w_d)}}{0}{\infty}{t} \leq \min\limits_{\ell = 1, \ldots, d} w_\ell
			\end{cases} 
			= \Lambda^+(w_0, \ldots, w_d) ~.
\end{align*}
The $(d+1)$-increasing property of $\phi$ needs to be verified on every rectangle $R \subset \R_+^{d+1}$. 
A direct calculation identical to that of Proposition~2.2 in \cite{Durante.2008} with $w_\ell^1 \leq w_\ell^2$ yields 
\begin{align*}
	V_\phi \rbraces{\bigtimes\limits_{\ell = 0}^d [w_\ell^1, w_\ell^2]}	
		&= \cInt{V_C \rbraces{\bigtimes\limits_{\ell = 1}^d \cbraces{\partial_1 \Lambda_\ell(t, w^1_\ell), \partial_1 \Lambda_\ell(t, w^2_\ell)}}}{w_0^1}{w_0^2}{t} \geq 0 ~,
\end{align*}
which holds due to $\partial_1 \Lambda_\ell(t, w^1_\ell) \leq \partial_1 \Lambda_\ell(t, w^2_\ell)$.
Lastly, the positive homogeneity can be established via a change of variables and the positive homogeneity of order $0$ of the partial derivatives of $\TDLn$,
\begin{align*}
	\phi(s\boldm{w})
	&= \cInt{C \rbraces{\partial_1 \Lambda_1(t, sw_1), \ldots, \partial_1 \Lambda_d(t, sw_d)}}{0}{sw_0}{t}\\
	&= \cInt{C \rbraces{\partial_1 \Lambda_1(t/s, w_1), \ldots, \partial_1 \Lambda_d(t/s, w_d)}}{0}{sw_0}{t} \\
	&= s \cInt{C \rbraces{\partial_1 \Lambda_1(z, w_1), \ldots, \partial_1 \Lambda_d(z, w_d)}}{0}{w_0}{z}
	= s \phi(\boldm{w}) ~.
\end{align*}
By Proposition~\ref{prop:tail_dep_func}, we can thus find a copula $C^*$ with $\TDL{\boldm{w}}{C^*} = \phi(\boldm{w})$ for all $\boldm{w} \in \R_+^{d+1}$. \\
The proof that $\Mprod{\Lambda_1, \ldots, \Lambda_d}{C}$ is a $d$-variate tail dependence function works analogously.
\end{proof}

\begin{remark} \label{remark:markov_product_subdistribution}
Note that the first part of the proof only requires that all $\Lambda_\ell$ are $2$-increasing functions bounded from below by 0 and from above by $\TDLn^+$.
Furthermore, $\phi$ is positive homogeneous of order one if and only if $\partial_1 \TDL{(t, w_\ell)}{C_\ell}$ is homogeneous of order zero for all $\ell = 1, \ldots, d$ . 
\end{remark}

Before considering the properties of the Markov product in more detail, let us discuss how to construct copulas with a given tail dependence function $\Lambda$.
For $d=2$, every tail dependence function defines an extreme-value copula via
\begin{equation} \label{eqn:extreme_value_copula_def}
	C^{EV}(u, v)	= \exp \rbraces{\log(u) + \log(v) + \Lambda(-\log(u), -\log(v))} ~. 
\end{equation}
Then the survival copula of $C^{EV}$ has lower tail dependence function $\Lambda$, see, for example, \cite{Gudendorf.2010}. 
For $d > 2$, the general construction is more involved and can be found in the proof of Proposition~6 of \cite{Jaworski.2006}.
Moreover, we will investigate in Theorem~\ref{thm:connectionTDF_MP_lipschitz} and Theorem~\ref{thm:connectionTDF_MP_reduction} under which circumstances the Markov product of copulas $(C_1 * C_2)$ fulfils 
\begin{equation*}
	\TDL{\boldm{w}}{C_1 * C_2} = (\TDL{\cdot}{C_1} * \TDL{\cdot}{C_2})(\boldm{w}) ~.
\end{equation*}

The next proposition compiles basic algebraic properties of $\Mprod{\Lambda_1, \ldots, \Lambda_d}{C}$ and $\Mgprod{\Lambda_1, \ldots, \Lambda_d}{C}{w_0}$.

\begin{proposition} \label{prop:algebraic_prop}
Suppose $\Lambda_1, \ldots, \Lambda_d$ are bivariate tail dependence functions and $C$ is a $d$-copula. 
Then
\begin{enumerate}
	\item $\TDLn^+ = \TDL{\cdot}{C^+}$ is the unit element in the sense that if $\Lambda_\ell = \TDLn^+$, then 
	\begin{equation*}
		\Mgprod{\Lambda_1, \ldots, \Lambda_d}{C}{w_0} (w_1, \ldots, w_d)	= \Mgprod{\Lambda_1, \ldots, \Lambda_{\ell - 1}, \Lambda_{\ell + 1}, \ldots, \Lambda_d}{\widehat{C}_\ell}{\min \braces{w_0, w_\ell}} (\boldm{\widehat{w}_\ell}) ~,
	\end{equation*}
	where $\boldm{\widehat{w}_\ell} := (w_1, \ldots, w_{\ell - 1}, w_{\ell+1}, \ldots, w_d)$ and $\widehat{C}_\ell := C(u_1, \ldots, u_{\ell - 1}, 1, u_{\ell + 1}, \ldots, u_d)$.
	\item $\TDLprod := \TDL{\cdot}{\Pi}$ is the null element in the sense that if $\Lambda_\ell = \TDLprod$, then
	\begin{equation*}
		\Mgprod{\Lambda_1, \ldots, \Lambda_d}{C}{w_0} = \TDL{(w_0, \ldots, w_d)}{\Pi^{d+1}} = 0 ~.
	\end{equation*} 
	\item If $C$ is convex resp.\ concave in the $\ell$-th component, then $\Mgprod{\cdot}{C}{w_0}$ is convex resp.\ concave in the $\ell$-th component.
	\item For every permutation $\pi$ on $\braces{1, \ldots, d}$, we have
	\begin{equation*}
		\Mprod{\Lambda_1, \ldots, \Lambda_d}{C}(w_{\pi(1)}, \ldots, w_{\pi(d)}) = 
			\Mprod{\Lambda_{\pi(1)}, \ldots, \Lambda_{\pi(d)}}{C^\pi}(w_1, \ldots, w_d) ~,
	\end{equation*}
	where $C^\pi(u_1,\ldots,u_d) := C(u_{\pi(1)},\ldots,u_{\pi(d)})$. \label{prop:algebraic_prop_perm}
	\item If $C \leq D$ pointwise, then 
	\begin{equation*}
		\Mgprod{\Lambda_1, \ldots, \Lambda_d}{C}{w_0} \leq \Mgprod{\Lambda_1, \ldots, \Lambda_d}{D}{w_0} ~.
	\end{equation*}
\end{enumerate}
\end{proposition}

\begin{remark}
The term \enquote{unit-element} stems from the bivariate case, where $\phi_C: \Md{2} \times \Md{2} \rightarrow \Md{2}$ constitutes a genuine product, which fulfils 
\begin{equation*}
	\Mprod{\TDLn^+, \Lambda}{C} (w_1, w_2)	= \Lambda(w_1, w_2)  ~.
\end{equation*}
\end{remark}

\begin{proof}
\begin{enumerate}
	\item Without loss of generality, we consider $\ell = 1$. 
	As $\partial_1 \TDL{(t, w_1)}{C^+} = \indFunc{[0, w_1]}(t)$, we have
	\begin{align*}
		\Mgprod{\Lambda_1, \ldots, \Lambda_d}{C}{w_0} (w_1, \ldots, w_d)	&= \cInt{C(1, \partial_1 \Lambda_2(t, w_2), \ldots, \partial_1 \Lambda_d(t, w_d))}{0}{\min\braces{w_0,w_1}}{t} \\
																		&= \cInt{\widehat{C}_1(\partial_1 \Lambda_2(t, w_2), \ldots, \partial_1 \Lambda_d(t, w_d))}{0}{\min\braces{w_0,w_1}}{t} \\
																		&= \Mgprod{\Lambda_2, \ldots, \Lambda_d}{\widehat{C}_1}{\min\braces{w_0,w_1}} (w_2, \ldots, w_d) ~.
	\end{align*}
	\item The second result is obvious since $\TDL{\cdot}{\Pi} \equiv 0$ and $C(0, u) = 0$. 
	\item The third result follows immediately from the pointwise inequality of $C$. 
	\item A direct calculation yields
	\begin{align*}
		 \Mprod{\Lambda_1, \ldots, \Lambda_d}{C}(w_{\pi(1)}, \ldots, w_{\pi(d)})	
						&= \cInt{C \rbraces{\partial_1 \Lambda_1(t, w_{\pi(1)}), \ldots, \partial_1 \Lambda_d(t, w_{\pi(d)})}}{0}{\infty}{t} \\
						&= \cInt{C^\pi \rbraces{\partial_1 \Lambda_{\pi(1)}(t, w_1), \ldots, \partial_1 \Lambda_{\pi(d)}(t, w_d)}}{0}{\infty}{t} \\
						&= \Mprod{\Lambda_{\pi(1)}, \ldots, \Lambda_{\pi(d)}}{C^\pi}(w_1, \ldots, w_d) ~.
	\end{align*}
\end{enumerate}
\end{proof}

Likewise, we will now collect some convergence results of the Markov product.
In the case of $2$-copulas, \cite{Siburg.2008} and \cite{Trutschnig.2011} introduced different metrics allowing for the joint convergence of the (general) Markov product for $2$-copulas, i.e. 
\begin{equation*}
	\phi_C(C_1^n, C_2^n) \rightarrow \phi_C(C_1, C_2) ~.
\end{equation*}
We will present very similar conditions in the case of tail dependence functions. 
Note, however, that due to the different domains of copulas and tail dependence functions, convergence of the partial derivatives coincides with the definition for $L^1$-convergence of the partial derivatives in the case of copulas and with pointwise convergence in the case of tail dependence functions.

\begin{proposition}
Suppose $\Lambda_1, \ldots, \Lambda_d$ are bivariate tail dependence functions and $C$ is a $d$-copula. 
\begin{enumerate}
	\item Let $C_n \in \Cd{d}$ with $C_n \rightarrow C$ pointwise. 
	Then
	\begin{equation*}
		\Mgprod{\Lambda_1, \ldots, \Lambda_d}{C_n}{w_0} \rightarrow \Mgprod{\Lambda_1, \ldots, \Lambda_d}{C}{w_0}
	\end{equation*}
	pointwise.
	\item Let $\Lambda_i^n \in \Md{2}$ with $\Lambda_i^n \rightarrow \Lambda_i$ pointwise. 
	Then
	\begin{equation*}
		\Mgprod{\Lambda^n_1, \ldots, \Lambda^n_d}{C}{w_0} \rightarrow \Mgprod{\Lambda_1, \ldots, \Lambda_d}{C}{w_0}
	\end{equation*}  
	pointwise.
	\item Let $\Lambda_i^n \in \Md{2}$ with $\norm{\partial_1 \Lambda_i^n(\cdot, w) - \partial_1 \Lambda_i(\cdot, w)}_{L^1(\R_+)} \rightarrow 0$, then 
	\begin{equation*}
		\Mprod{\Lambda^n_1, \ldots, \Lambda^n_d}{C} \rightarrow \Mprod{\Lambda_1, \ldots, \Lambda_d}{C}
	\end{equation*}
	pointwise.
\end{enumerate}
\end{proposition}

\begin{proof}
\begin{enumerate}
	\item A combination of $C(\partial_1 \Lambda_1 (t, w_1), \ldots, \partial_1 \Lambda_d (t, w_d)) \leq \partial_1 \Lambda_1(t, w_1)$ and the dominated convergence theorem yields the desired result.
	\item Due to the Lipschitz continuity of $C$, we have
	\begin{align*}
		\abs{\Mgprod{\Lambda^n_1, \ldots, \Lambda^n_d}{C}{w_0}(w) - \Mgprod{\Lambda_1, \ldots, \Lambda_d}{C}{w_0}(w)}	&\leq \sum\limits_{i = 1}^d \cInt{\abs{\partial_1 \Lambda_i^n(t, w_i) - \partial_1 \Lambda_i(t, w_i)}}{0}{w_0}{t} ~.
	\end{align*}
	Thus it suffices to consider each integral separately. 
	As tail dependence functions are concave, Lemma~1 in \cite{Tsuji.1952} yields that $\Lambda_i^n(t, w_i) \rightarrow \Lambda_i (t,w_i)$ holds pointwise if and only if $\partial_1 \Lambda_i^n(t, w_i) \rightarrow \partial_1 \Lambda_i (t,w_i)$ holds for almost all $t \in [0, w_0]$ and all fixed $w \in \R_+$.
	Thus, an application of the dominated convergence theorem in combination with $0 \leq \partial_1 \Lambda_i^n \leq 1$ yields the desired result.
	\item Again, due to the Lipschitz continuity of $C$, we have 
	\begin{equation*}
		\abs{\Mprod{\Lambda^n_1, \ldots, \Lambda^n_d}{C}(w) - \Mprod{\Lambda_1, \ldots, \Lambda_d}{C}(w)}	\leq \sum\limits_{i = 1}^d \cInt{\abs{\partial_1 \Lambda_i^n(t, w_i) - \partial_1 \Lambda_i(t, w_i)}}{0}{\infty}{t} \rightarrow 0 ~. \qedhere
	\end{equation*}
\end{enumerate}
\end{proof}

In analogy to the binary product $*$ on $\Cd{2} \times \Cd{2}$ induced by $\Pi$, we introduce $*$ on $\Md{2} \times \Md{2}$ via
\begin{equation*}
	(\Lambda_1 * \Lambda_2) (w_1, w_2) 	:= \Mprod{\Lambda_1^T, \Lambda_2}{\Pi}(w_1, w_2)
										= \cInt{\partial_2 \Lambda_1 (w_1, t) \partial_1 \Lambda_2(t, w_2)}{0}{\infty}{t} ~.
\end{equation*}
Its properties closely resemble those of the Markov product on $\Cd{2} \times \Cd{2}$.
In particular, $\TDLn^+$ and $\TDLprod$ are the unit and null element of $*$, respectively, and $*$ is associative as well as skew-symmetric, i.e.
\begin{equation*}
	(\Lambda_1 * \Lambda_2)^T = \Lambda_2^T * \Lambda_1^T ~. 
\end{equation*}
With these basic algebraic properties, we will develop two conditions under which the Markov product commutes with the tail dependence function, i.e.
\begin{equation*}
	\TDL{\boldm{w}}{C_1 * C_2} = \TDL{\cdot}{C_1} * \TDL{\cdot}{C_2} (\boldm{w}) ~.
\end{equation*}
The first approach utilizes the Lipschitz continuity of $C$ and follows an idea from \cite{Jaworski.2015}.
Theorem~7 therein derives the tail behaviour of the $C$-lifting 
\begin{equation*}
	\TDL{(w_0, \ldots, w_d)}{\Mgprod{C_1, \ldots, C_d}{C}{\cdot}} = \Mgprod{\TDL{\cdot}{C_1}, \ldots, \TDL{\cdot}{C_d}}{C}{w_0}(w_0, \ldots, w_d)
\end{equation*}
under a Sobolev-type condition imposed on $C_1, \ldots, C_d$.

\begin{theorem} \label{thm:connectionTDF_MP_lipschitz}
Suppose that $C$ is a $d$-copula and that $C_1, \ldots, C_d$ are $2$-copulas with existing bivariate tail dependence functions, which fulfil the Sobolev-type condition
\begin{equation} \label{eqn:sobolevtypeCondition}
	\lim\limits_{s \searrow 0} \cInt{\abs{\partial_1 C_i(s t, s w) \indFunc{\cbraces{0, \frac{1}{s}}}(t) - \partial_1 \TDL{(t, w)}{C_i}}}{0}{\infty}{t} = 0
\end{equation}
for all $w \in \R_+$ and all $i = 1, \ldots, d$. 
Then,
\begin{equation*}
	\Mprod{\TDL{\cdot}{C_1}, \ldots, \TDL{\cdot}{C_d}}{C}(\boldm{w}) = \TDL{\boldm{w}}{\Mprod{C_1, \ldots, C_d}{C}}
\end{equation*}
for all $\boldm{w} \in \R_+^d$, or, equivalently,
\begin{figure}[h]
\centering
	\begin{tikzcd}
\Cd{2}^d \arrow[r, "\phi_C"] \arrow[d, "\Lambda(\cdot \; ; \; C_i)"'] \arrow[dr, phantom, "\circlearrowleft"] 	& \Cd{d} \arrow[d, "\Lambda(\cdot \; ; \; C_i)"] \\
\Md{2}^d \arrow[r, "\phi_C"']                                         				& \Md{d}
\end{tikzcd} 
\end{figure}
\end{theorem}

\begin{proof}
The Lipschitz continuity and groundedness of $C$ yield
\begin{align*}
&\abs{C \rbraces{\partial_1 C_1(s\tau, sw_1), \ldots, \partial_1 C_d(s\tau, sw_d)} \indFunc{\cbraces{0, \frac{1}{s}}}(\tau) - 
				C \rbraces{\partial_1 \TDL{(\tau, w_1)}{C_1}, \ldots, \partial_1 \TDL{(\tau, w_d)}{C_d}}} \\
&\leq \sum\limits_{\ell = 1}^{d} \abs{\partial_1 C_\ell(s\tau, sw_\ell)\indFunc{\cbraces{0, \frac{1}{s}}}(\tau) - \partial_1 \TDL{(\tau, w_\ell)}{C_\ell}} ~.
\end{align*}
Thus, 
\begin{align*}	
  &\abs{\TDL{\boldm{w}}{\Mprod{C_1, \ldots, C_d}{C}} - \Mprod{\TDL{\cdot}{C_1}, \ldots, \TDL{\cdot}{C_d}}{C}(\boldm{w})} \\
	\leq \lim\limits_{s \searrow 0} 	&\sum\limits_{\ell = 1}^d \lebesgue{\abs{\partial_1 C_\ell(s\tau, sw_\ell)\indFunc{\cbraces{0, \frac{1}{s}}}(\tau) - \partial_1 \TDL{(\tau, w_\ell)}{C_\ell}}}{\R_+}{\tau}
			= 0 ~. \qedhere
\end{align*}
\end{proof}

Using the concept of strict tail dependence functions yields a more feasible sufficient condition for Theorem~\ref{thm:connectionTDF_MP_lipschitz}.
We call a tail dependence function strict if it has margins in the sense of \cite{Nelsen.2006}.

\begin{definition}
Let $\TDLn$ be a bivariate tail dependence function. 
We call $\TDLn$ strict if 
\begin{equation*}
	\lim\limits_{t \rightarrow \infty} \TDLn(w_1, t) = w_1 \; \text{ and } \; \lim\limits_{t \rightarrow \infty} \TDLn(t, w_2) = w_2
\end{equation*}
holds for all $(w_1, w_2)$ in $\R_+^2$.
\end{definition}

\begin{remark} \label{remark:cor_thm_commutation_mp}
Assume that in addition to the almost everywhere pointwise convergence of the partial derivatives, the tail dependence functions of $C_i$ are strict. 
Then an application of Scheffé's Lemma (see, \cite{Novinger.1972}) yields 
\begin{equation*}
	\lim\limits_{s \searrow 0} \cInt{\abs{\partial_1 C_i(s t, s w) \indFunc{\cbraces{0, \frac{1}{s}}}(t) - \partial_1 \TDL{(t, w)}{C_i}}}{0}{\infty}{t} = 0 
\end{equation*}
for all $i = 1, \ldots, d$, which implies 
\begin{equation*}
	\Mprod{\TDL{\cdot}{C_1}, \ldots, \TDL{\cdot}{C_d}}{C}(\boldm{w}) = \TDL{\boldm{w}}{\Mprod{C_1, \ldots, C_d}{C}} 
\end{equation*}
due to Theorem~\ref{thm:connectionTDF_MP_lipschitz}.
\end{remark}

\begin{example}
Suppose $C_\phi$ is an Archimedean $2$-copula with generator $\phi$, which is regularly varying in $0$ with parameter $-\alpha < 0$.
Then its tail dependence function equals
\begin{equation*}
	\TDL{\boldm{w}}{C_\phi}	= \begin{cases}
							\rbraces{w_1^{-\alpha} + w_i^{-\alpha}}^{-1/\alpha}		&\text{ if } \alpha \in (0, \infty)	\\
							\TDLn^+(\boldm{w})										&\text{ if } \alpha = \infty
					  \end{cases} 
\end{equation*}
see, for example, \cite{Charpentier.2009}.
Thus, $\TDL{\cdot}{C_\phi}$ is strict. 
The convergence of the partial derivatives follows similar to the proof of Proposition~2.5 in \cite{Joe.2010}, when combined with the limiting behaviour
\begin{equation*}
	\lim\limits_{x \rightarrow \infty} \frac{\phi'(x)}{c \alpha x^{\alpha - 1} \ell(x)} = 1
\end{equation*}
given in Theorem~1.7.2 of \cite{Bingham.1987}, where $\ell$ is a slowly varying function.
Therefore, Archimedean copulas fulfil the conditions given in Remark~\ref{remark:cor_thm_commutation_mp}.
\end{example}

The next approach does not utilize the Lipschitz continuity of the copula $C$ and yields a different condition in terms of the convergence of the partial derivatives. 

\begin{theorem} \label{thm:connectionTDF_MP_reduction}
Suppose $C$ is a $d$-copula and that the $2$-copulas $C_1, \ldots, C_d$ as well as their generalized Markov product have a tail dependence function. 
If for all $w_i \in \R_+$ and almost all $t \in \R_+$
\begin{equation*}
	\lim\limits_{s \searrow 0} \partial_1 C_i(s t, s w_i) = \partial_1 \TDL{(t, w_i)}{C_i} 	\text{ for all } i = 1, \ldots, d ~, 
\end{equation*}
then 
\begin{equation*}
	\Mprod{\TDL{\cdot}{C_1}, \ldots, \TDL{\cdot}{C_d}}{C}(\boldm{w}) \leq \TDL{\boldm{w}}{\Mprod{C_1, \ldots, C_d}{C}} ~.
\end{equation*}
Additionally, if there exists an $\ell \in \braces{1, \ldots, d}$ such that
\begin{equation*}
	\partial_1 C_k(s \tau, s w_\ell) \indFunc{\cbraces{0, \frac{1}{s}}}(\tau)  \leq g_{w_\ell}(\tau)
\end{equation*} 
for all $w_\ell \in [0, 1]$ and some family $(g_w)_{w \in [0, 1]}$ of integrable functions, it holds
\begin{equation*}
	\Mprod{\TDL{\cdot}{C_1}, \ldots, \TDL{\cdot}{C_d}}{C}(\boldm{w}) = \TDL{\boldm{w}}{\Mprod{C_1, \ldots, C_d}{C}} ~.
\end{equation*}
\end{theorem}

\begin{proof}
By the definition of the tail dependence function and an application of Fatou's lemma for positive measurable functions, it holds that
\begin{align*}
	\TDL{\boldm{w}}{\Mprod{C_1, \ldots, C_d}{C}}	&= \lim\limits_{s \searrow 0} \frac{1}{s} \cInt{C \rbraces{\partial_1 C_1(t, sw_1), \ldots, \partial_1 C_d(t, sw_d)}}{0}{1}{t} \\
											&= \lim\limits_{s \searrow 0} \frac{1}{s} \cInt{C \rbraces{\partial_1 C_1(s\tau, sw_1), \ldots, \partial_1 C_d(s\tau, sw_d)}s}{0}{\frac{1}{s}}{\tau} \\
											&= \lim\limits_{s \searrow 0} \lebesgue{C \rbraces{\partial_1 C_1(s\tau, sw_1), \ldots, \partial_1 C_d(s\tau, sw_d)}\indFunc{\cbraces{0, \frac{1}{s}}}(\tau)}{\R_+}{\tau} \\
											&\geq \lebesgue{\lim\limits_{s \searrow 0} C \rbraces{\partial_1 C_1(s\tau, sw_1), \ldots, \partial_1 C_d(s\tau, sw_d)}\indFunc{\cbraces{0, \frac{1}{s}}}(\tau)}{\R_+}{\tau} \\
											&= \lebesgue{C \rbraces{\partial_1 \TDL{(\tau, w_1)}{C_1}, \ldots, \partial_1 \TDL{(\tau, w_d)}{C_d}}}{\R_+}{\tau}  \\
											&=\Mprod{\TDL{\cdot}{C_1}, \ldots, \TDL{\cdot}{C_d}}{C}(\boldm{w}) ~.
\end{align*}
If one partial derivative $\ell$ is dominated by an integrable function $g_{w_\ell}$ we have that for $\tau \leq 1/s$
\begin{align*}
	C \rbraces{\partial_1 C_1(s\tau, sw_1), \ldots, \partial_1 C_d(s\tau, sw_d)}
		&\leq C^+ \rbraces{\partial_1 C_1(s\tau, sw_1), \ldots, \partial_1 C_d(s\tau, sw_d)} \\
		&\leq \partial_1 C_\ell(s \tau, s w_\ell) \leq g_{w_\ell}(\tau) ~.
\end{align*}
The desired result follows from the dominated convergence theorem. 
\end{proof}

\begin{example}
If $C$ has a tail dependence function and continuous second-order partial derivatives, then \cite{Joe.2010} have shown that
\begin{equation*}
	\lim\limits_{s \searrow 0} \partial_1 C(s t, s w) = \partial_1 \TDL{(t, w)}{C} 	
\end{equation*}
holds for almost all $t$ and all $w$ in $\R_+$, thus fulfilling the condition of Theorem~\ref{thm:connectionTDF_MP_reduction}.
\end{example}

\begin{example}
Suppose $C$ is the survival copula of an extreme value copula $C^{EV}$ given by $\Lambda \in \Md{2}$, see Equation~\eqref{eqn:extreme_value_copula_def}. 
Then $C$ has tail dependence function $\Lambda$ and a straight-forward calculation using the positive homogeneity of degree $0$ of $\partial_1 \Lambda$ yields
\begin{align*}
	\partial_1 C(su, sv)	&= 1 - \partial_1 C^{EV}(1 - su, 1-sv)	\\
							&= 1 - \frac{C^{EV}(1-su, 1-sv)}{1 - su} \rbraces{1 - \partial_1 \Lambda (-\log(1-su), -\log(1-sv))} \\
							&= 1 - \underbrace{\frac{C^{EV}(1-su, 1-sv)}{1 - su}}_{\rightarrow 1 \text{ as } s \searrow 0} \rbraces{1 - \partial_1 \Lambda \rbraces{1, \underbrace{\frac{\log(1-sv)}{\log(1-su)}}_{\rightarrow v/u \text{ as } s \searrow 0 }}}~.
\end{align*}
As $t \mapsto \partial_1 \Lambda(w, t)$ is increasing and bounded from below by 0 and from above by $1$, we have
\begin{equation*}
	\lim\limits_{s \searrow 0} \partial_1 C(su, sv)	= \partial_1 \Lambda \rbraces{1, \frac{v}{u}} = \partial_1 \Lambda \rbraces{u, v} ~.
\end{equation*}
\end{example}

The lower bound behaviour stated in Theorem~\ref{thm:connectionTDF_MP_reduction} is generally the best result possible, as can be seen from the following example. 

\begin{example}
Consider the lower Fréchet-Hoeffding bound $C^-$, which is symmetric and left invertible, i.e., $(C^-)^T * C^- = C^+$.
Then an application of Theorem 5.5.3 in \cite{Durante.2015} yields
\begin{equation*}
	\Mprod{C^-, C^-}{C} = C^- * C^- = (C^-)^T * C^- = C^+ ~.
\end{equation*}
Hence for $\boldm{w} = (w_1, w_2)$ in $\R_+^2$,
\begin{equation*}
	\Mprod{\TDL{\cdot}{C^-}, \TDL{\cdot}{C^-}}{C}(\boldm{w}) = 0 \leq \min\braces{w_1, w_2} = \TDL{\boldm{w}}{C^+} = \TDL{w}{\Mprod{C^-, C^-}{C}} ~,
\end{equation*} 
which is strict for every $\boldm{w}$ in $(0, \infty)^2$. 
\end{example}

Let us now study some examples to investigate the behaviour of the Markov product on $\Md{2}$ for different $2$-copulas $C$.

\begin{example}
Let $C \in \Cd{d}$ and $\Lambda_1, \ldots, \Lambda_d \in \M^+$, where
\begin{align*}
	\mathcal{M}^+ 	&:= \set{\TDLn}{\partial_1 \TDLn (w) = \alpha \indFunc{\cbraces{0, \frac{\beta}{\alpha} w_2}}(w_1) \text{ for some } \alpha, \beta \in (0, 1]} \\
					&\ = \set{\TDLn}{\TDLn (w) = \min\braces{\alpha w_1, \beta w_2} \text{ for some } \alpha, \beta \in (0, 1]} ~.
\end{align*}
Then 
\begin{align*}
	\Mprod{\Lambda_1, \ldots, \Lambda_d}{C}(\boldm{w})	&= C \rbraces{\alpha_1, \ldots, \alpha_d} \TDL{\frac{\beta_1}{\alpha_1} w_1, \ldots, \frac{\beta_d}{\alpha_d} w_d}{C^+} ~.
\end{align*}
The influence of the choice of $C$ on the product is depicted in Figure~\ref{fig:prodSimSim}.
For the two tail dependence functions $\Lambda_1(w_1, w_2) = \min\braces{\frac{2w_1}{3}, w_2}$ and $\Lambda_2(w_1, w_2) = \min\braces{\frac{w_1}{2}, \frac{w_2}{4}}$, the resulting product $\Mprod{\Lambda_1, \Lambda_2}{C}$ is shown by the red line for the choices $C=C^-$, $C=\Pi$ and $C = C^+$, respectively.\\
\begin{figure}
	\centering
	\begin{subfigure}[b]{0.32\textwidth}
          \centering
          \resizebox{\linewidth}{!}{\begin{tikzpicture}[declare function={ 
		tdf(\t) 	= min(\t, 1-\t); 
		tdf1(\t) 	= min(2*\t/3, 1-\t);
		tdf2(\t) 	= min(0.5*\t, 0.25*(1-\t));
		cop(\u,\v)	= max(\u + \v - 1, 0);
		tdf12C(\t) 	= cop(2/3, 0.5)*min(3/2*\t, 0.5*(1-\t));
}]
        \begin{axis}[mlineplot, width=7cm,height=5cm]
            \addplot[name path=f, domain=0:1, samples=101, color=gray!30]{tdf(\x)};
            \addplot[name path=f, domain=0:1, samples=101, color=black!60]{tdf1(\x)};
            \addplot[name path=f, domain=0:1, samples=101, color=black!60]{tdf2(\x)};
            \addplot[name path=f, domain=0:1, samples=101, color=red!30]{tdf12C(\x)};
        \end{axis}
\end{tikzpicture}}
          \caption{$\Mprod{\Lambda_1, \Lambda_2}{C^-}$.}
          \label{fig:prodSimSimA}
     \end{subfigure}
     \begin{subfigure}[b]{0.32\textwidth}
          \centering
          \resizebox{\linewidth}{!}{
\begin{tikzpicture}[
	declare function={ 
		tdf(\t) 	= min(\t, 1-\t); 
		tdf1(\t) 	= min(2*\t/3, 1-\t);
		tdf2(\t) 	= min(0.5*\t, 0.25*(1-\t));
		cop(\u,\v)	= \u*\v;
		tdf12C(\t) 	= cop(2/3, 0.5)*min(3/2*\t, 0.5*(1-\t));
}]
        \begin{axis}[mlineplot, width=7cm,height=5cm]
            \addplot[name path=f, domain=0:1, samples=101, color=gray!30]{tdf(\x)};
            \addplot[name path=f, domain=0:1, samples=101, color=black!60]{tdf1(\x)};
            \addplot[name path=f, domain=0:1, samples=101, color=black!60]{tdf2(\x)};
            \addplot[name path=f, domain=0:1, samples=101, color=red!30]{tdf12C(\x)};
        \end{axis}
\end{tikzpicture}}
          \caption{$\Mprod{\Lambda_1, \Lambda_2}{\Pi}$.}
          \label{fig:prodSimSimB}
     \end{subfigure}
     \begin{subfigure}[b]{0.32\textwidth}
          \centering
          \resizebox{\linewidth}{!}{
\begin{tikzpicture}[
	declare function={ 
		tdf(\t) 	= min(\t, 1-\t); 
		tdf1(\t) 	= min(2*\t/3, 1-\t);
		tdf2(\t) 	= min(0.5*\t, 0.25*(1-\t));
		cop(\u,\v)	= min(\u, \v);
		tdf12C(\t) 	= cop(2/3, 0.5)*min(3/2*\t, 0.5*(1-\t));
}]
        \begin{axis}[mlineplot, width=7cm,height=5cm]
            \addplot[name path=f, domain=0:1, samples=101, color=gray!30]{tdf(\x)};
            \addplot[name path=f, domain=0:1, samples=101, color=black!60]{tdf1(\x)};
            \addplot[name path=f, domain=0:1, samples=101, color=black!60]{tdf2(\x)};
            \addplot[name path=f, domain=0:1, samples=101, color=red!30]{tdf12C(\x)};
        \end{axis}
\end{tikzpicture}}
          \caption{$\Mprod{\Lambda_1, \Lambda_2}{C^+}$.}
          \label{fig:prodSimSimC}
     \end{subfigure}
     \caption{Plots of the product $\Mprod{\Lambda_1, \Lambda_2}{C}(t, 1-t)$ for different choices of $C$ (red line) following Remark~\ref{remark:bivariate_homogeneity}.
     The tail dependence functions $\Lambda_1(t, 1-t) = \min\braces{\frac{2t}{3}, 1-t}$ and $\Lambda_2(t, 1-t) = \min\braces{\frac{t}{2}, \frac{1-t}{4}}$ are depicted in black, the upper bound $\Lambda^+$ in grey.}
     \label{fig:prodSimSim}
\end{figure}
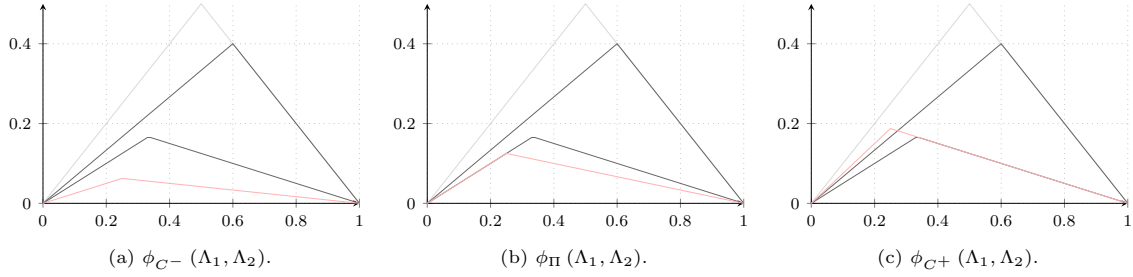
Taking the product of $\Lambda_1 \in \M^+$ and an arbitrary $\Lambda_2 \in \Md{2}$  yields
\begin{equation*}
	\Mprod{\Lambda_1, \Lambda_2}{C}(w_1, w_2) = \cInt{C(\alpha, \partial_1 \Lambda_2(t, w_2))}{0}{\frac{\beta}{\alpha} w_1}{t} ~.
\end{equation*} 
The above expression can be explicitly calculated  for some choices of $C$ (see, Figure~\ref{fig:prodSimCop}): 
\begin{enumerate}
	\item If $C = \Pi$, then $\Mprod{\Lambda_1, \Lambda_2}{\Pi}(w_1, w_2) = \Lambda_2 \rbraces{\beta w_1, \alpha w_2}$. 
	\item If $C = C^-$, we have
	\begin{equation*}
		\Mprod{\Lambda_1, \Lambda_2}{C^-}(w_1, w_2)	= \Lambda_2(p^* \wedge \beta w_1, w_2) + \TDL{\alpha p^*, \beta w_1}{C^+} - p^* ~,
	\end{equation*} 
	since the monotonicity of $\partial_1 \Lambda_2$ yields the existence of a $p^* = p^*(\alpha, w_2) \geq 0$  with
	\begin{equation*}
		\partial_1 \Lambda_2(t, w_2) + \alpha -1 \geq 0 \text{ for all } t \leq p^* \text{ and } \partial_1 \Lambda_2(t, w_2) + \alpha - 1 \leq 0 \text{ for all } t > p^* ~.
	\end{equation*}
	\item By a similar argument, for $C = C^+$, it holds
	\begin{align*}
	\Mprod{\Lambda_1, \Lambda_2}{C^+}(w_1, w_2)	&= \TDL{\alpha p^*, \beta w_1}{C^+} + \Lambda_2 \rbraces{\frac{\beta}{\alpha} w_1, w_2} - \Lambda_2 \rbraces{p^* \wedge \frac{\beta}{\alpha} w_1, w_2} ~,
	\end{align*}
	where $p^* = p^*(1-\alpha, w_2)$. 
\end{enumerate}
\begin{figure}
	\centering
	\begin{subfigure}[b]{0.32\textwidth}
          \centering
          \resizebox{\linewidth}{!}{\begin{tikzpicture}[ declare function={ 
    			tdf(\t) 	= min(\t, 1-\t); 
				tdf1(\t) 	= min(0.5*\t, 1-\t);
				tdf2(\t) 	= min(\t, 0.25*(1-\t));
				tdf3(\x,\y)	= \x*\y/(\x + \y);
				tdf3St(\t)	= (1/sqrt(1-0.5) - 1)*(1-\t);
				tdf13m(\t)	= tdf3(min(tdf3St(\t), \t), 1-\t) + min(0.5*min(tdf3St(\t), \t), 0.5*\t) - min(tdf3St(\t), \t);
			}]
        \begin{axis}[mlineplot, width=7cm,height=5cm]
            \addplot[name path=f, domain=0:1, samples=101, color=gray!30]{tdf(\x)};
            \addplot[name path=f, domain=0:1, samples=151, color=black!60]{tdf1(1-\x)};
            \addplot[name path=f, domain=0:1, samples=151, color=black!60]{tdf3(\x, 1-\x)};
            \addplot[name path=f, domain=0:1, samples=151, color=red!30]{tdf13m(1-\x)};
        \end{axis}
\end{tikzpicture}}
          \caption{$\Mprod{\Lambda_1, \Lambda_2}{C^-}$.}
          \label{fig:prodSimCopA}
     \end{subfigure}
	\begin{subfigure}[b]{0.32\textwidth}
          \centering
          \resizebox{\linewidth}{!}{\begin{tikzpicture}[ declare function={ 
    			tdf(\t) 	= min(\t, 1-\t); 
				tdf1(\t) 	= min(0.5*\t, 1-\t);
				tdf2(\t) 	= min(\t, 0.25*(1-\t));
				tdf3(\x,\y)	= \x*\y/(\x + \y);
				tdf13(\t)	= tdf3(\t, 0.5*(1-\t));
			}]
        \begin{axis}[mlineplot, width=7cm,height=5cm]
            \addplot[name path=f, domain=0:1, samples=101, color=gray!30]{tdf(\x)};
            \addplot[name path=f, domain=0:1, samples=151, color=black!60]{tdf1(1-\x)};
            \addplot[name path=f, domain=0:1, samples=151, color=black!60]{tdf3(\x, 1-\x)};
            \addplot[name path=f, domain=0:1, samples=151, color=red!30]{tdf13(1-\x)};
        \end{axis}
\end{tikzpicture}}
          \caption{$\Mprod{\Lambda_1, \Lambda_2}{\Pi}$.}
          \label{fig:prodSimCopB}
     \end{subfigure}
     \begin{subfigure}[b]{0.32\textwidth}
          \centering
          \resizebox{\linewidth}{!}{\begin{tikzpicture}[ declare function={ 
    			tdf(\t) 	= min(\t, 1-\t); 
				tdf1(\t) 	= min(0.5*\t, 1-\t);
				tdf2(\t) 	= min(\t, 0.25*(1-\t));
				tdf3(\x,\y)	= \x*\y/(\x + \y);
				tdf3St(\t)	= (1/sqrt(0.5) - 1)*(1-\t);
				tdf13p(\t)	= min(0.5*tdf3St(\t), \t) + tdf3(2*\t, 1-\t) - tdf3(min(tdf3St(\t), 2*\t), 1-\t);
			}]
        \begin{axis}[mlineplot, width=7cm,height=5cm]
            \addplot[name path=f, domain=0:1, samples=101, color=gray!30]{tdf(\x)};
            \addplot[name path=f, domain=0:1, samples=151, color=black!60]{tdf1(1-\x)};
            \addplot[name path=f, domain=0:1, samples=151, color=black!60]{tdf3(\x, 1-\x)};
            \addplot[name path=f, domain=0:1, samples=151, color=red!30]{tdf13p(1-\x)};
        \end{axis}
\end{tikzpicture}}
          \caption{$\Mprod{\Lambda_1, \Lambda_2}{C^+}$.}
          \label{fig:prodSimCopC}
     \end{subfigure}
     \caption{Plots of the product $\Mprod{\Lambda_1, \Lambda_2}{C}(t, 1-t)$ for different choices of $C$ (red line) following Remark~\ref{remark:bivariate_homogeneity}. 
     The tail dependence functions $\Lambda_1(t, 1-t) = \min\braces{\frac{t}{2}, 1-t}$ and $\Lambda_2(t, 1-t) = t(1-t)$ are depicted in black, the upper bound $\Lambda^+$ in grey.}
     \label{fig:prodSimCop}
\end{figure}
\end{example}

\section{Monotonicity of the Markov product} \label{section:monotonicity}

Figures~\ref{fig:prodSimSim} and \ref{fig:prodSimCop} already suggest a monotonicity of the Markov product whenever $C$ fulfils a negative dependence property. 
We will treat this property in more detail in this section. 
\begin{theorem} \label{thm:dependence_reduction}
Let $\Lambda_1, \ldots, \Lambda_d \in \Md{2}$ and $C \in \Cd{d}$ be negatively quadrant dependent, i.e. $C \leq \Pi$. 
Then, for $k = 1, \ldots, d$ and $\boldm{w} \in \R_+^d$,
\begin{equation*}
	\Mprod{\Lambda_1, \ldots, \Lambda_d}{C}(\boldm{w})	\leq \Mprod{\Lambda_1, \ldots, \Lambda_d}{\Pi} (\boldm{w})
												\leq \min\limits_{\substack{m = 1, \ldots, d\\ m \neq k}} \Lambda_k (w_m, w_k) ~.
\end{equation*}
\end{theorem}

This result decidedly contrasts with the behaviour of the Markov product for $2$-copulas, where for example
\begin{equation*}
	 C^-  \leq C^+ = C^- * C^- ~.
\end{equation*}
Theorem~\ref{thm:dependence_reduction} is incorrect without the assumption that $C \leq \Pi$, as can be seen in Figure~\ref{fig:prodSimCop}(c). 
We will give two different proofs of Theorem~\ref{thm:dependence_reduction}; the second proof is deferred to Section~\ref{section:operator} since it uses the theory of substochastic operators developed there.

\begin{proof}
Due to $\Lambda_1 \leq \Lambda^+$, we have
\begin{equation*}
	\cInt{\partial_1 \Lambda_1(s, w_1)}{0}{t}{s} \leq \cInt{\partial_1 \Lambda^+(s, w_1)}{0}{t}{s}
\end{equation*}
for all $w_1, t \in [0, \infty)$. 
Hardy's Lemma (see, \cite{Bennett.1988}) yields for any positive decreasing function $f: \R_+ \rightarrow \R_+$ that
\begin{equation*}
	\cInt{\partial_1 \Lambda_1(s, w_1) f(s)}{0}{\infty}{s} 	\leq \cInt{\partial_1 \Lambda^+(s, w_1) f(s)}{0}{\infty}{s} 
															= \cInt{f(s)}{0}{w_1}{s} ~.
\end{equation*}
Thus, for all tail dependence functions $\Lambda_2, \ldots, \Lambda_d$ and any $\boldm{w} \in \R_+^d$
\begin{align*}
	\Mprod{\Lambda_1, \ldots, \Lambda_d}{\Pi}(\boldm{w})	&= \cInt{\partial_1 \Lambda_1(s, w_1) \partial_1 \Lambda_2 (s, w_2) \cdots  \partial_1 \Lambda_d (s, w_d)}{0}{\infty}{s} \\
															&\leq \cInt{\partial_1 \Lambda_2 (s, w_2) \cdots  \partial_1 \Lambda_d (s, w_d)}{0}{w_1}{s} \\
															&= \Mgprod{\Lambda_2, \ldots, \Lambda_d}{\Pi}{w_1}(w_2, \ldots, w_d) ~.
\end{align*}
An application of \ref{prop:algebraic_prop_perm}. in Proposition~\ref{prop:algebraic_prop}. yields the desired result. 
\end{proof}

\begin{corollary}
Let $C$ be an idempotent $2$-copula, i.e. $C * C = C$.
Then for all $\boldm{w} \in \R_+^2$,
\begin{equation*}
	\TDL{\boldm{w}}{C * C} \geq \TDL{\cdot}{C} * \TDL{\cdot}{C} (\boldm{w}) ~.
\end{equation*}
\end{corollary}

\begin{proof}
Theorem~\ref{thm:dependence_reduction} in combination with $C * C = C$ immediately yields
\begin{equation*}
	\TDL{\boldm{w}}{C * C} = \TDL{\boldm{w}}{C} \geq \TDL{\cdot}{C} * \TDL{\cdot}{C} (\boldm{w}) ~. \qedhere
\end{equation*}
\end{proof}

Theorem~\ref{thm:dependence_reduction} can be strengthened for bivariate tail dependence functions at zero and at one. 
Due to the concavity of tail dependence functions and Remark~\ref{remark:bivariate_homogeneity}, an application of Theorem~\ref{thm:dependence_reduction} yields
\begin{equation*}
	(\TDLph{\Lambda_1 * \Lambda_2})'(0)	= \lim\limits_{s \searrow 0} \frac{(\TDLph{\Lambda_1 * \Lambda_2)}(s)}{s}
								\leq \lim\limits_{s \searrow 0} \frac{\min\braces{\TDLph{\Lambda_1}(s), \TDLph{\Lambda_2}(s)}}{s}
								= \min \braces{\TDLph{\Lambda_1}'(0) , \TDLph{\Lambda_2}'(0)} ~,
\end{equation*}
despite of Figure~\ref{fig:reductionProperty} suggesting a much stronger result.

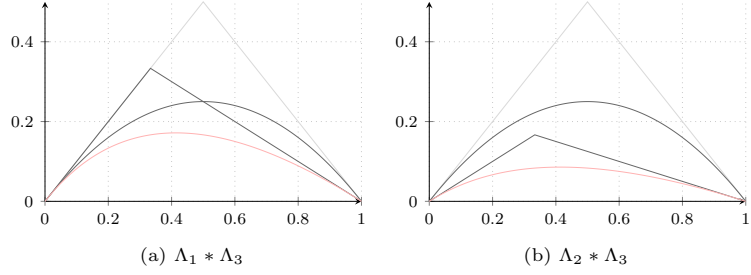
\begin{figure}
	\centering
	\begin{subfigure}[b]{0.32\textwidth}
          \centering
          \resizebox{\linewidth}{!}{\begin{tikzpicture}[ declare function={ 
    			tdf(\t) 	= min(\t, 1-\t); 
				tdf1(\t) 	= min(0.5*\t, 1-\t);
				tdf2(\t) 	= min(\t, 0.25*(1-\t));
				tdf3(\x,\y)	= \x*\y/(\x + \y);
				tdf13(\t)	= tdf3(\t, 0.5*(1-\t));
			}]
        \begin{axis}[mlineplot, width=7cm,height=5cm]
            \addplot[name path=f, domain=0:1, samples=101, color=gray!30]{tdf(\x)};
            \addplot[name path=f, domain=0:1, samples=151, color=black!60]{tdf1(1-\x)};
            \addplot[name path=f, domain=0:1, samples=151, color=black!60]{tdf3(\x, 1-\x)};
            \addplot[name path=f, domain=0:1, samples=151, color=red!30]{tdf13(\x)};
        \end{axis}
\end{tikzpicture}}
          \caption{$\Lambda_1 * \Lambda_3$}
          \label{fig:reductionProdA}
    \end{subfigure}
	\begin{subfigure}[b]{0.32\textwidth}
          \centering
          \resizebox{\linewidth}{!}{\begin{tikzpicture}[ declare function={ 
    			tdf(\t) 	= min(\t, 1-\t); 
				tdf1(\t) 	= 0.5*min(0.5*\t, 1-\t);
				tdf2(\t) 	= min(\t, 0.25*(1-\t));
				tdf3(\x,\y)	= \x*\y/(\x + \y);
				tdf13(\t)	= 0.5*tdf3(\t, 0.5*(1-\t));
			}]
        \begin{axis}[mlineplot, width=7cm,height=5cm]
            \addplot[name path=f, domain=0:1, samples=101, color=gray!30]{tdf(\x)};
            \addplot[name path=f, domain=0:1, samples=151, color=black!60]{tdf1(1-\x)};
            \addplot[name path=f, domain=0:1, samples=151, color=black!60]{tdf3(\x, 1-\x)};
            \addplot[name path=f, domain=0:1, samples=151, color=red!30]{tdf13(\x)};
        \end{axis}
\end{tikzpicture}}
          \caption{$\Lambda_2 * \Lambda_3$}
          \label{fig:reductionProdB}
     \end{subfigure}
     \caption{Plots of the products $\Lambda_1 * \Lambda_2 (t, 1-t)$ and $\Lambda_2 * \Lambda_3 (t, 1-t)$ (red line) following Remark~\ref{remark:bivariate_homogeneity}. 
     The tail dependence functions $\Lambda_1(t, 1-t) = \min \rbraces{\frac{t}{2}, 1-t}, \Lambda_2(t, 1-t) =  \frac{1}{2} \min \rbraces{\frac{t}{2}, 1-t}$ and $\Lambda_3(t, 1-t) =t(1-t)$ are depicted in black, the upper bound $\Lambda^+$ in grey.}
     \label{fig:reductionProperty}
\end{figure}

\begin{proposition} \label{prop:derivative_factorization}
For $\Lambda_1, \Lambda_2 \in \Md{2}$, it holds 
\begin{equation*}
	(\TDLph{\Lambda_1 * \Lambda_2})'(0) = \TDLph{\Lambda}_1'(0) \cdot \TDLph{\Lambda}_2'(0) \text{ and } (\TDLph{\Lambda_1 * \Lambda_2})'(1) = - \TDLph{\Lambda}_1'(1) \cdot \TDLph{\Lambda}_2'(1) ~.
\end{equation*}
Furthermore, for any negatively quadrant dependent $C \in \Cd{2}$, i.e. $C \leq \Pi$, 
\begin{equation*}
	-\TDLph{\Lambda}_1'(1) \cdot \TDLph{\Lambda}_2'(1) \leq (\TDLph{\Lambda_1 *_C \Lambda_2})'(t) \leq  \TDLph{\Lambda}_1'(0) \cdot \TDLph{\Lambda}_2'(0) 
\end{equation*}
for all $t \in [0, 1]$. 
\end{proposition}

\begin{proof}
The positive homogeneity of a tail dependence function $\Lambda$ implies
\begin{equation*}
	\Lambda(x, y)	= (x + y) \TDLphFo{\Lambda} \rbraces{\frac{x}{x+y}} ~,
\end{equation*} 
for all $(x + y) > 0$, which in turn yields 
\begin{equation*}
	\partial_1 \Lambda(x, y) = \TDLphFo{\Lambda} \rbraces{\frac{x}{x+y}} + \frac{y}{x+y} \TDLphFo{\Lambda}' \rbraces{\frac{x}{x+y}} ~.
\end{equation*}
Thus, for any $y > 0$, we obtain
\begin{equation*}
	\partial_1 \Lambda(0, y) = \TDLphFo{\Lambda}' \rbraces{0} ~.
\end{equation*}
An application of the product rule for the Stieltjes integral yields
\begin{align*}
	(\TDLph{\Lambda_1 * \Lambda_2})'(0)	&= \partial_1 \rbraces{\Lambda_1 * \Lambda_2} (0, y) 
										= \left. \partial_x \cInt{\partial_2 \Lambda_1(x, t) \partial_1 \Lambda_2(t, y)}{0}{\infty}{t} \right|_{x=0} \\
										&= \left.\rInt{\partial_1 \Lambda_2(t, y)}{0}{\infty}{ \partial_1 \Lambda_1(x, \mathrm{d} t)} \right|_{x=0} \\
										&= \left. \partial_1 \Lambda_2(\infty, y) \partial_1 \Lambda_1(x, \infty) - \partial_1 \Lambda_2(0, y) \partial_1 \Lambda_1(x, 0) 
											- \rInt{\partial_1 \Lambda_1(x, t)}{0}{\infty}{\partial_1 \Lambda_2(\mathrm{d} t, y)} \right|_{x=0} \\
										&= - \rInt{\partial_1 \Lambda_1(0, t)}{0}{\infty}{\partial_1 \Lambda_2(\mathrm{d} t, y)} 
										= - \TDLph{\Lambda_1}'(0) \rInt{1}{0}{\infty}{\partial_1 \Lambda_2(\mathrm{d} t, y)} \\
										&= - \TDLph{\Lambda_1}'(0) \cbraces{\partial_1 \Lambda_2(\infty, y) - \partial_1 \Lambda_2(0, y)} 
										= \TDLph{\Lambda_1}'(0) \cdot \TDLph{\Lambda_2}'(0) ~,
\end{align*}
where the third equality can be shown analogously to Lemma 3.1 of \cite{Darsow.1992}. 
The second claim can be derived by observing that $(\TDLph{\Lambda_1 * \Lambda_2})'(1) = (\TDLph{\Lambda_2^T * \Lambda_1^T})'(0)$.
Finally, the last assertion stems from the fact that $\TDLph{\Lambda_1 * \Lambda_2}$ is concave and thus has a monotone derivative.
\end{proof}

The next lemma shows that Proposition~\ref{prop:derivative_factorization} has a surprising influence on the global behaviour of $\TDLn_1 * \TDLn_2$, by connecting the derivative of $\TDLph{\TDLn_1 * \TDLn_2}$ in zero to the behaviour of $(\TDLn_1 * \TDLn_2)$ towards $\infty$.

\begin{lemma} \label{lemma:tdf_strict}
Let $\Lambda$ be a bivariate tail dependence function.
Then the following are equivalent
\begin{enumerate}
	\item $\TDLph{\Lambda}'(0) = a \in [0, 1]$. \label{lemma:tdf_strict_derivative}
	\item $\lim\limits_{y \rightarrow \infty} \Lambda(x, y) = ax$ for all $x \in \R_+$. \label{lemma:tdf_strict_boundary}
\end{enumerate}
\end{lemma}

\begin{proof}
As tail dependence functions are concave, the right-hand derivative $\TDLph{\Lambda}'(0)$ exists. 
Then, for all $x \in \R_+$, it holds that
\begin{align*}
	\lim\limits_{y \rightarrow \infty} \Lambda(x, y)	&= \lim\limits_{y \rightarrow \infty} (x + y) \Lambda \rbraces{\frac{x}{x + y}, \frac{y}{x+y}} \\
														&= \lim\limits_{y \rightarrow \infty} x \frac{x+y}{x} \TDLph{\Lambda} \rbraces{\frac{x}{x + y}} 
														= \lim\limits_{t \rightarrow 0} x \frac{\TDLph{\Lambda}(t)}{t}
														= x \TDLph{\Lambda}'(0) ~. \qedhere
\end{align*}
\end{proof}

The factorization of $\TDLphFo{\Lambda_1 * \Lambda_2}$ is only valid in $0$ and $1$ and does not generally hold for $s \in (0, 1)$, i.e. $(\TDLph{\Lambda_1 * \Lambda_2})'(s) \neq \TDLph{\Lambda}_1'(s) \cdot \TDLph{\Lambda}_2'(s)$, see for example Figure~\ref{fig:reductionProperty}.
Nevertheless, a general smoothing property concerning the Markov product can be derived, which is reminiscent of \cite{Trutschnig.2013b}.  

\begin{theorem}
Suppose $\Lambda_1, \Lambda_2 \in \Md{2}$. 
Then $\TDLphFo{\Lambda_1 * \Lambda_2}$ is differentiable if $\TDLphFo{\Lambda_1}$ or $\TDLphFo{\Lambda_2}$ is differentiable. 
\end{theorem}

\begin{proof}
First, the derivative of $\TDLphFo{\Lambda_1 * \Lambda_2}$ can be rewritten as
\begin{equation*}
	(\TDLphFo{\Lambda_1 * \Lambda_2})'(s)	= \partial_1 \rbraces{\Lambda_1 * \Lambda_2}(s, 1-s) - \partial_2 \rbraces{\Lambda_1 * \Lambda_2}(s, 1-s) ~.
\end{equation*}
We will treat both terms on the right-hand side separately.
First, 
\begin{align*}
	 \partial_1 \rbraces{\Lambda_1 * \Lambda_2}(s, 1-s)	&=  \partial_1 \cInt{\partial_2 \Lambda_1(s, t) \partial_1 \Lambda_2(t, 1-s)}{0}{\infty}{t}	\\
	 													&= \rInt{\partial_1 \Lambda_2(t, 1-s)}{0}{\infty}{\partial_1 \Lambda_1(s, \mathrm{d}t) } ~,
\end{align*}
where w.l.o.g. $\partial_1 \Lambda_1(s, t)$ exists for all $s \in [0, 1]$ and is increasing in $t$, otherwise switch the roles of $\Lambda_1$ and $\Lambda_2$ due to symmetry. 
The second equality can again be shown identically to Lemma 3.1 of \cite{Darsow.1992}.
Analogously, 
\begin{align*}
	 \partial_2 \rbraces{\Lambda_1 * \Lambda_2}(s, 1-s)	&= \rInt{\partial_1 \Lambda_2(t, 1-s)}{0}{\infty}{\partial_1 \Lambda_1(s, \mathrm{d}t)} \\
	 													&= - \rInt{\partial_2 \Lambda_2(t, 1-s)}{0}{\infty}{\partial_2 \Lambda_1(s, \mathrm{d}t)} ~. \qedhere
\end{align*}
\end{proof}

While the inverses with respect to the Markov product for $2$-copulas can be used to analyse complete dependence and extremal points of $\Cd{2}$, the reduction property impedes an analogy for tail dependence functions.   

\begin{theorem}
Suppose $\Lambda \in \Md{2}$.
\begin{enumerate}
	\item If $\Lambda$ is left-invertible, i.e. there exists a bivariate tail dependence function $\xi$ such that  $\xi * \Lambda (\boldm{w}) = \TDL{\boldm{w}}{C^+}$, then $\Lambda(\boldm{w}) = \TDL{\boldm{w}}{C^+}$.
	\item If $\partial_1 \Lambda(w_1, w_2) \in \braces{0, 1}$ for almost all $w_2 \in \R_+$, then  $\Lambda(w_1, w_2) = \TDL{w_1, \alpha w_2}{C^+}$ for some $\alpha \in [0, 1]$.
\end{enumerate}
\end{theorem}

\begin{proof}
\begin{enumerate}
	\item If $\Lambda$ is left-invertible with left-inverse $\xi$, then 
\begin{equation*}
	\TDL{\boldm{w}}{C^+} = \xi * \Lambda(\boldm{w}) \leq \Lambda(\boldm{w}) \leq \TDL{\boldm{w}}{C^+} ~.
\end{equation*}
	\item Assuming $\Lambda$ is a tail dependence function with $\partial_1 \Lambda(w_1, w_2) \in \braces{0, 1}$ for almost all $w_2 \in \R_+$, then there exists a function $\alpha: [0, \infty) \rightarrow [0, 1]$ such that
	\begin{equation*}
		\partial_1 \Lambda(w_1, w_2) = \indFunc{[0, \alpha(w_2) w_2)]}(w_1) ~. 
	\end{equation*}
	The positive homogeneity of $\Lambda$ implies that $\partial_1 \Lambda$ is positive homogeneous of order $0$, i.e. constant along rays. 
	Thus, for all $s > 0$ this leads to
	\begin{align*}
		\indFunc{[0, \alpha(sw_2)w_2]}(w_1) &= \indFunc{[0, \alpha(sw_2)sw_2]}(sw_1)  
											= \partial_1 \Lambda(sw_1, sw_2) \\
											&= \partial_1 \Lambda(w_1, w_2) 
											= \indFunc{[0, \alpha(w_2)w_2]}(w_1) ~.
	\end{align*}
	Consequently, $\alpha(sw_2) = \alpha(w_2) = \alpha$. \qedhere
\end{enumerate}
\end{proof}

Lastly, we derive a monotonicity property of the Markov product with respect to the pointwise order of tail dependence functions. 

\begin{corollary} \label{cor:markov_product_monotonicity}
For $\Lambda_1, \Lambda_2 \in \Md{2}$, the following are equivalent: 
\begin{enumerate}
	\item $\Lambda_1(\boldm{w}) \leq \Lambda_2(\boldm{w}) $ for all $\boldm{w} \in \R_+^2$. \label{cor:markov_product_monotonicity_order}
	\item $(\Lambda_1 * \Lambda)(\boldm{w}) \leq (\Lambda_2 * \Lambda)(\boldm{w})$ for all $w \in \R_+^2$ and $\Lambda \in \Md{2}$. \label{cor:markov_product_monotonicity_product}
\end{enumerate}
\end{corollary}

\begin{proof}
The implication \ref{cor:markov_product_monotonicity_product}. to \ref{cor:markov_product_monotonicity_order}. follows immediately from the choice $\Lambda = \Lambda^+$. 
Conversely, assuming $\Lambda_1(\boldm{w}) \leq \Lambda_2(\boldm{w})$ for all $\boldm{w} \in \R_+^2$, we have
\begin{equation*}
	\cInt{\partial_2 \Lambda_1(w_1, t)}{0}{w_2}{t}	= \Lambda_1(\boldm{w}) \leq \Lambda_2(\boldm{w})	= \cInt{\partial_2 \Lambda_2(w_1, t)}{0}{w_2}{t} ~.
\end{equation*}
Since $\partial_1 \Lambda(\cdot, w_2)$ is non-negative and decreasing for any tail dependence function $\Lambda \in \Md{2}$, Proposition 2.3.6 in \cite{Bennett.1988} yields
\begin{equation*}
	(\Lambda_1 * \Lambda)(\boldm{w})		= \cInt{\partial_2 \Lambda_1(w_1, t) \partial_1 \Lambda(t, w_2)}{0}{\infty}{t} 
											\leq  \cInt{\partial_2 \Lambda_2(w_1, t) \partial_1 \Lambda(t, w_2) }{0}{\infty}{t}
											= (\Lambda_2 * \Lambda)(\boldm{w}) ~. \qedhere
\end{equation*}
\end{proof}
\section{Iterates of the Markov product} \label{section:iterates}

In the context of $2$-copulas, the concepts of iterates, idempotents, and Cesàro sums of the Markov product are widely investigated, see, for example, \cite{Darsow.2010} or \cite{Trutschnig.2013b}. 
To investigate these concepts in the setting of tail dependence functions, we define the $n$-th iterate of the Markov product for $2$-copulas and tail dependence functions as
\begin{equation*}
	C^{*n} := \underbrace{C * \ldots * C}_{n\text{ times}} \text{ and } \TDLn^{*n} :=  \TDLn * \ldots * \TDLn ~,
\end{equation*}
respectively. 
\cite{Trutschnig.2013} showed the existence of the Cesàro sums
\begin{equation*}
	\widehat{C} := \lim\limits_{n \rightarrow \infty} \frac{1}{n} \sum\limits_{\ell = 1}^n C^{*\ell} 
\end{equation*}
for general $2$-copulas $C$ and treated their limit behaviour using ergodic theory. 
Here, we will study the asymptotic behaviour of $\Lambda^{*n}$ and extend the results to an averaging of the Markov product. 
First, we will develop an understanding using two simple examples. 

\begin{example}
Consider a copula $C$ such that 
\begin{equation*}
	\partial_1 \TDL{\boldm{w}}{C} = \indFunc{[0, \alpha w_2]}(w_1) \text{ with } \alpha \in [0, 1] ~.
\end{equation*}
A simple calculation yields
\begin{equation*}
	\TDL{\cdot}{C}^{* 2} (\boldm{w})	= \cInt{\partial_2 \TDL{(w_1, t)}{C} \indFunc{[0, \alpha w_2]} (t) }{0}{\infty}{t} 
										= \TDL{(w_1, \alpha w_2)}{C}
\end{equation*}
and iteratively
\begin{equation*}
	\TDL{\cdot}{C}^{* n} (\boldm{w})	= \TDL{(w_1, \alpha^{n-1} w_2)}{C} \rightarrow \begin{cases} 0		& \text{ for } \alpha \in [0, 1) \\ \min \braces{w_1, w_2}		&\text{ for } \alpha = 1 \end{cases} ~.
\end{equation*}
Thus, in this example, the limiting behaviour of $\Lambda^{*n}$ is either given by $\TDLprod$ or $\TDLn^+$. 
\end{example}

The next example treats a class of tail dependence functions, which will be utilized to dominate arbitrary tail dependence functions and ultimately characterize idempotents.

\begin{example}
For $0 \leq p \leq \frac{1}{2}$ define the function
\begin{equation} \label{eqn:tdf_dom}
	\TDLph{\TDLn}(s) := 
		\begin{cases} 
			s	&\text{ for } 0 \leq s \leq p \\
			p	&\text{ for } p \leq s \leq 1-p \\
			1-s	&\text{ for } 1-p \leq s \leq 1
		\end{cases} ~.
\end{equation}
Following Remark~\ref{remark:bivariate_homogeneity}, $\Lambda$ can be extended to a tail dependence function on $\R_+^2$. 
A straightforward calculation with $q:= \frac{1-p}{p}$ yields the recurrence equation
\begin{equation*}
	\TDLn^{* (n+1)}(w_1, w_2)	= (1-p) \TDLn^{* n}\rbraces{\frac{1}{q} w_1, w_2}	+ p \TDLn^{* n} \rbraces{q w_1, w_2} ~.
\end{equation*}
It can be solved in two steps.
First, it holds
\begin{equation*}
	\TDLn^{* (n+1)}(w_1, w_2)	= \sum\limits_{\ell = 0}^{n} a_\ell^{n} \TDLn \rbraces{q^{n - 2\ell}w_1, w_2}
\end{equation*}
with $a_\ell^n \in \R_+$ such that
\begin{equation*}
	a_0^0 = 1 \; , \; a_0^{n+1} = p^{n}	\text{ and } a_\ell^{n+1}	= (1-p) a_{\ell - 1}^n + p a_\ell^n \text{ for } 1 \leq \ell \leq n ~.
\end{equation*}
The general solution of multivariate recurrences of this type was derived by \cite{Neuwirth.2001} and \cite{Mansour.2013} and is  given by
\begin{equation*}
	a_\ell^n	= \binom{n}{\ell} (1-p)^\ell p^{n - \ell}	\text{ for } 0 \leq \ell \leq n ~. 
\end{equation*} 
Using the positive homogeneity of $\Lambda$, we arrive at the solution
\begin{equation*}
	\TDLn^{* (n+1)}(w_1, w_2)	= p^{n} \sum\limits_{\ell = 0}^{n} \binom{n}{\ell} \TDLn \rbraces{q^{n-\ell}w_1, q^{\ell} w_2} ~.
\end{equation*}
An example of the behaviour of $\TDLn^{* n}$ is shown in Figure~\ref{fig:iterativeProduct} for different $n$ and $p=\frac{1}{3}$.
We will now derive the asymptotic behaviour of $\Lambda^{*n}$ for $n \rightarrow \infty$. 
Due to the iterated Markov product being symmetric and due to the monotonicity of $*$, it suffices to consider $w_1 = w_2 = \frac{1}{2}$ and uneven $n = 2k+1$. 
It holds
\begin{align*}
	\TDLn^{* (2k+1)}\rbraces{\frac{1}{2}, \frac{1}{2}}	&= \frac{p^{2k}}{2} \sum\limits_{\ell = 0}^{2k} \binom{2k}{\ell} \TDLn \rbraces{q^{2k-\ell}, q^{\ell}} \\
														&= \frac{p^{2k}}{2} \sum\limits_{\ell = 0}^{2k} \binom{2k}{\ell} \rbraces{q^{2k-\ell} + q^{\ell}} \TDLph{\TDLn} \rbraces{\frac{q^{2k-\ell}}{q^{2k-\ell} + q^{\ell}}} \\
														&\leq p^{2k} \sum\limits_{\ell = 0}^{k} \binom{2k}{\ell} q^\ell - p^{2k+1}q^{k} \binom{2k}{k} \\
														&= \sum\limits_{\ell = 0}^{k} \binom{2k}{\ell} (1-p)^\ell p^{2k-\ell} - \binom{2k}{k} p^{k+1}(1-p)^{k} ~,  
\end{align*}
where the inequality is due to the definition of $\Lambda(s)$ and equality holds in case of $p = 1/2$. 
While the second part converges to zero as $n \rightarrow \infty$, the first part is a truncated binomial sum and by the weak law of large numbers, we have
\begin{align*}
\lim\limits_{k \rightarrow \infty} \max\limits_{w_1 + w_2 = 1} \TDLn^{* (2k+1)}\rbraces{w_1, w_2}
 								&= \lim\limits_{k \rightarrow \infty} \TDLn^{* (2k+1)}\rbraces{\frac{1}{2}, \frac{1}{2}} \\	
								&\leq \lim\limits_{k \rightarrow \infty}  \sum\limits_{\ell = 0}^{k} \binom{2k}{\ell} (1-p)^\ell p^{2k-\ell} - \binom{2k}{k} p^{k+1}(1-p)^{k}  \\
								&= \begin{cases}	0	&\text{ for } p < \frac{1}{2}\\ \frac{1}{2}	&\text{ for } p = \frac{1}{2} \end{cases} ~.
\end{align*}
Due to $0 \leq \Lambda^{*(2k+1)}$, the above inequality is in fact an equality. 
\begin{figure}
	\centering
	\begin{subfigure}[b]{0.24\textwidth}
          \centering
          \resizebox{\linewidth}{!}{ \begin{tikzpicture}[ declare function={ 
    			tdf(\t) 	= min(\t, 1-\t); 
				f(\t)		= min(\t, 1/3, 1-\t);
			}]
        \begin{axis}[mlineplot]
            \addplot[domain=0:1, 	samples=101, color=gray!30]{tdf(\x)};
            \addplot[domain=0:1, 	samples=151, color=black!60]{f(\x)};
        \end{axis}
\end{tikzpicture}}
          \caption{$\TDLph{\TDLn}(s)$}
          \label{fig:iterProdA}
    \end{subfigure}
	\begin{subfigure}[b]{0.24\textwidth}
          \centering
          \resizebox{\linewidth}{!}{ \begin{tikzpicture}[ declare function={ 
    			tdf(\t) 	= min(\t, 1-\t); 
				f(\x,\y)	= (2*\x<=\y)?\x:((2*\y<=\x)?\y:(\x+\y)/3);
				h(\x,\y)	= (f(2*\x,\y)+f(\x,2*\y))/3;
				g(\x)		= h(\x,1-\x);
		}]
        \begin{axis}[mlineplot]
            \addplot[domain=0:1, 	samples=101, color=gray!30]{tdf(\x)};
			\addplot[domain=0:1, 	samples=101, color=gray!30]{f(\x,1-\x)};
            \addplot[domain=0:1, 	samples=301, color=black!60]{g(\x)};
        \end{axis}
\end{tikzpicture}}
          \caption{$\TDLph{\TDLn^{* 2}}(s)$}
          \label{fig:iterProdB}
     \end{subfigure}
     \begin{subfigure}[b]{0.24\textwidth}
          \centering
          \resizebox{\linewidth}{!}{ \begin{tikzpicture}[ declare function={ 
    			tdf(\t) 	= min(\t, 1-\t); 
				f(\x,\y)	= (2*\x<=\y)?\x:((2*\y<=\x)?\y:(\x+\y)/3);
				h(\x,\y)	= (f(2*\x,\y)+f(\x,2*\y))/3;
				g(\x,\y)	= (f(4*\x,\y)+2*f(2*\x,2*\y)+1*f(\x,4*\y))/9;	
		}]
        \begin{axis}[mlineplot]
            \addplot[domain=0:1, 	samples=101, color=gray!30]{tdf(\x)};
			\addplot[domain=0:1, 	samples=101, color=gray!30]{h(\x,1-\x)};
            \addplot[domain=0:1, 	samples=301, color=black!60]{g(\x,1-\x)};
        \end{axis}
\end{tikzpicture}}
          \caption{$\TDLph{\TDLn^{* 3}}(s)$}
          \label{fig:iterProdC}
     \end{subfigure}
     \begin{subfigure}[b]{0.24\textwidth}
          \centering
          \resizebox{\linewidth}{!}{ \begin{tikzpicture}[ declare function={ 
    			tdf(\t) 	= min(\t, 1-\t); 
				f(\x,\y)	= (2*\x<=\y)?\x:((2*\y<=\x)?\y:(\x+\y)/3);
				h(\x,\y)	= (f(8*\x,\y)+6*f(2*\x,\y)+6*f(\x,2*\y)+f(\x,8*\y))/27;	
				g(\x,\y)	= (f(4*\x,\y)+2*f(2*\x,2*\y)+1*f(\x,4*\y))/9;	
				k(\x,\y)	= (f(16*\x,\y)+4*f(8*\x,2*\y)+6*f(4*\x,4*\y)+4*f(2*\x,8*\y)+1*f(\x,16*\y))/81;		
		}]
        \begin{axis}[mlineplot]
            \addplot[domain=0:1, 	samples=101, color=gray!30]{tdf(\x)};
			\addplot[domain=0:1, 	samples=101, color=gray!30]{h(\x,1-\x)};
			\addplot[domain=0:1, 	samples=101, color=gray!30]{g(\x,1-\x)};
            \addplot[domain=0:1, 	samples=301, color=black!60]{k(\x,1-\x)};
        \end{axis}
\end{tikzpicture}}
          \caption{$\TDLph{\TDLn^{* 5}}(s)$}
          \label{fig:iterProdD}
     \end{subfigure}
     \caption{Plots of the tail dependence function $\Lambda$ from Equation~\eqref{eqn:tdf_dom} and its iterations $\Lambda^{*n}$ for $n=2, 3$ and $5$ and $p=\frac{1}{3}$.}
     \label{fig:iterativeProduct}
\end{figure}
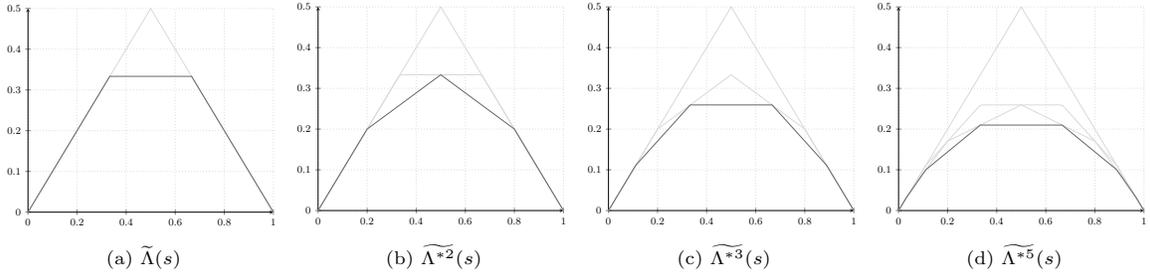
\end{example}

Using the monotonicity property of the Markov product from Corollary~\ref{cor:markov_product_monotonicity} and the fact that the previous examples dominate any tail dependence function, we arrive at the following result. 

\begin{theorem}
Let $\Lambda$ be a bivariate tail dependence function. 
Then 
\begin{equation*}
	\lim\limits_{n \rightarrow \infty} \Lambda^{* n}(\boldm{w})		= \begin{cases} 
																		\TDL{\boldm{w}}{C^+}	&\text{ for } \Lambda = \TDL{\cdot}{C^+} \\
																		\TDL{\boldm{w}}{\Pi}	&\text{ for } \Lambda \neq \TDL{\cdot}{C^+} 
																	  \end{cases} 
\end{equation*}
and the Cesàro sum equals
\begin{align*}
	\TDLn^*(\boldm{w}) 	&:= \lim\limits_{n \rightarrow \infty} \frac{1}{n} \sum\limits_{\ell = 1}^n \Lambda^{*\ell} (\boldm{w})
						= \lim\limits_{n \rightarrow \infty} \Lambda^{* n}(\boldm{w})	~.
\end{align*}
\end{theorem}

This result gives another indication that the Markov product has smoothing properties, as tail independence corresponds to Fréchet-differentiability of $C$ in zero. 

\begin{proof}
If $\Lambda = \TDLn^+$, the result is immediate. 
Thus, consider a tail dependence function $\Lambda$ with $\Lambda \neq \TDLn^+$.
Define 
\begin{equation*}
	p := \max\limits_{t \in [0, 1]} \TDLph{\Lambda}(t) < \frac{1}{2} ~.
\end{equation*}
and set 
\begin{equation*}
	\TDLph{\Lambda_p}(s) := 
		\begin{cases} 
			s	&, 0 \leq s \leq p \\
			p	&, p \leq s \leq 1-p \\
			1-s	&, 1-p \leq s \leq 1
		\end{cases} ~.
\end{equation*}
Thus, $\Lambda_p$ dominates $\Lambda$, i.e. $\Lambda \leq \Lambda_p$, and Corollary~\ref{cor:markov_product_monotonicity} yields by induction
\begin{equation*}
	\Lambda^{*n}(\boldm{w}) = \Lambda^{*(n-1)} * \Lambda (\boldm{w}) \leq \Lambda^{*(n-1)} * \Lambda_p (\boldm{w}) \leq \Lambda_p^{*n}(\boldm{w}) \rightarrow 0 
\end{equation*}
for any $p < \frac{1}{2}$.
For the second statement, we only need to verify that the partial Cesàro sums are decreasing. 
Applying the monotonicity of $*$ yields
\begin{align*}
	\frac{1}{n} \sum\limits_{\ell = 1}^n \Lambda^{*\ell} - \frac{1}{n+1} \sum\limits_{\ell = 1}^{n+1} \Lambda^{*\ell}
	&= \rbraces{\frac{1}{n} - \frac{1}{n+1}} \sum\limits_{\ell = 1}^{n} \Lambda^{*\ell} - \frac{1}{n+1} \Lambda^{*(n+1)}	\\
	&= \frac{1}{n(n+1)} \sum\limits_{\ell = 1}^{n} \Lambda^{*\ell} - \frac{1}{n+1} \Lambda^{*(n+1)}	\\
	&\geq \frac{1}{n(n+1)} \sum\limits_{\ell = 1}^{n} \Lambda^{* n} - \frac{1}{n+1} \Lambda^{* (n+1)} \\
	&= \frac{\Lambda^{* n} - \Lambda^{* (n+1)}}{n+1} \geq 0 ~.
\end{align*}
The limit of a mean of concave functions is again concave and bounded and thus a bivariate tail dependence function. 
Moreover, Dini's theorem implies that the monotone convergence of continuous functions on a compact set to a continuous function must be uniform, i.e.
\begin{equation*}
	\lim\limits_{n \rightarrow \infty} \norm{\frac{1}{n} \sum\limits_{\ell = 1}^n \TDLph{\Lambda^{*\ell}} - \TDLph{\TDLn^*}}_\infty = 0 ~. \qedhere
\end{equation*}

\end{proof}

This theorem has two immediate corollaries, one in regard to idempotent tail dependence functions, and the other to the connection to the tail behaviour of the generalized Markov product. 

\begin{corollary}
A bivariate tail dependence function $\Lambda \in \Md{2}$ is idempotent, i.e. $\Lambda * \Lambda = \Lambda$, if and only if $\Lambda=\TDLn^+$ or $\Lambda=\TDLprod$. 
\end{corollary}

\begin{proof}
If $\Lambda$ is idempotent, we have
\begin{equation*}
	\Lambda (\boldm{w}) = \lim\limits_{n \rightarrow \infty} \Lambda^{* n}(\boldm{w})		
						= 	\begin{cases} 
								\TDL{\boldm{w}}{C^+}	&\text{ for } \Lambda = \TDL{\cdot}{C^+} \\
								\TDL{\boldm{w}}{\Pi}	&\text{ for } \Lambda \neq \TDL{\cdot}{C^+} 
							\end{cases} ~. \qedhere
\end{equation*}
\end{proof}

Finally, we link the previous results to the tail behaviour of iterates and idempotents of $2$-copulas.

\begin{corollary}
Suppose $C$ is a twice continuously differentiable $2$-copula on $(0,1)^2$ with a strict tail dependence function.
If we define $\widehat{C}$ as the Cesàro sum of $C$ with respect to $*$, then 
\begin{equation*}
	\TDL{\boldm{w}}{\widehat{C}} 	= \lim\limits_{n \rightarrow \infty} \frac{1}{n} \sum\limits_{\ell = 1}^n \TDL{\boldm{w}}{C}^{*\ell} 
									= \begin{cases} 
											\TDL{\boldm{w}}{C^+}	&\text{ for } \TDL{\cdot}{C} = \TDL{\cdot}{C^+} \\
											\TDL{\boldm{w}}{\Pi}	&\text{ for } \TDL{\cdot}{C} \neq \TDL{\cdot}{C^+} 
						   			  \end{cases} ~.
\end{equation*}
\end{corollary}

\begin{proof}
Consider the tail dependence function of $\widehat{C}$, i.e.
\begin{align*}
	\TDL{\boldm{w}}{\widehat{C}}	&= \lim\limits_{s \searrow 0} \frac{\widehat{C}(s\boldm{w})}{s}
									= \lim\limits_{s \searrow 0} \lim\limits_{n \rightarrow \infty} \frac{1}{n} \sum\limits_{\ell = 1}^n \frac{C^{*\ell}(s\boldm{w})}{s} 
									=: \lim\limits_{s \searrow 0} \lim\limits_{n \rightarrow \infty} f(n, s) ~.
\end{align*}
Note that $f(n, s)$ converges pointwise for fixed $s$ as $n \rightarrow \infty$ as well as pointwise for fixed $n$ as $s \searrow 0$. 
Moreover, Theorem~2 in \cite{Trutschnig.2013} implies the uniform convergence of $\lim_n f(n, s)$. 
Thus, the iterated limit above can be interchanged and it holds
\begin{align*}
	\TDL{\boldm{w}}{\widehat{C}}	&= \lim\limits_{s \searrow 0} \lim\limits_{n \rightarrow \infty} \frac{1}{n} \sum\limits_{\ell = 1}^n \frac{C^{*\ell}(s\boldm{w})}{s} \\
									&= \lim\limits_{n \rightarrow \infty} \lim\limits_{s \searrow 0} \frac{1}{n} \sum\limits_{\ell = 1}^n \frac{C^{*\ell}(s\boldm{w})}{s} \\
									&= \lim\limits_{n \rightarrow \infty}\frac{1}{n} \sum\limits_{\ell = 1}^n \TDL{\boldm{w}}{C^{*\ell}} \\
									&= \lim\limits_{n \rightarrow \infty}\frac{1}{n} \sum\limits_{\ell = 1}^n \TDL{\boldm{w}}{C}^{* \ell} ~.
\end{align*}
The last equality stems from an inductive argument utilizing 
\begin{equation*}
	\TDL{\boldm{w}}{C * D} = \TDL{\cdot}{C} * \TDL{\cdot}{D}(\boldm{w})
\end{equation*}
for all twice differentiable $2$-copulas $C$ and $D$.
The result follows from observing that $C*D$ is again twice differentiable if $C$ and $D$ are twice differentiable, and strict if both tail dependence functions are strict. 
\end{proof}

\section{Substochastic operators} \label{section:operator}

We previously saw the close resemblance of the set of $2$-copulas endowed with the Markov-product and the set of bivariate tail dependence function endowed with $*$.
In case of the set of $2$-copulas, \cite{Olsen.1996} derived an isomorphy to integral-preserving linear operators. 
Along those lines, we will subsequently draw a connection between a certain class of linear operators and bivariate tail dependence functions. 
For this we define the underlying space
\begin{equation*}
	 L^1(\R_+) + L^\infty(\R_+) := \set{f + g}{f \in  L^1(\R_+) \text{ and } g \in L^\infty(\R_+)} 
\end{equation*}
and both $L^1(\R_+)$ and $L^\infty(\R_+)$ are subsets of $L^1(\R_+) + L^\infty(\R_+)$. 

\begin{definition} \label{def:substochastic_operator}
A linear operator $T: L^1(\R_+) + L^\infty(\R_+) \rightarrow L^1(\R_+) + L^\infty(\R_+)$ is called doubly substochastic if
\begin{enumerate}
	\item $T$ is positive, i.e. $Tf \geq 0$ whenever $f \geq 0$. \label{def:substochastic_operator_positive}
	\item $T(L^1(\R_+)) \subset L^1(\R_+)$ and $T(L^\infty(\R_+)) \subset L^\infty(\R_+)$. \label{def:substochastic_operator_mapping}
	\item $T$ is non-expansive on $L^1(\R_+)$ and $L^\infty(\R_+)$, respectively, i.e. $\norm{Tf}_1 \leq \norm{f}_1$ and $\norm{Tg}_\infty \leq \norm{g}_\infty$	for all $f \in L^1(\R_+)$ and $g \in L^\infty(\R_+)$.\label{def:substochastic_operator_contraction}
	\item $Tf = \sup_n Tf_n$ for all $f_n  \in L^1(\R_+) \cap L^\infty(\R_+) $ and $f \in L^\infty(\R_+)$ such that $f_n \nearrow f$. \label{def:substochastic_operator_continuation}
\end{enumerate}
$T$ is called equivariant if
\begin{equation*}
	T \rbraces{f \circ \sigma} = (Tf) \circ \sigma
\end{equation*}
holds for all dilations $\sigma (x) := \frac{x}{s}$ with $s > 0$. 
\end{definition}

Substochastic operators can be seen as a generalization of Markov operators, in the same way as doubly substochastic matrices generalize doubly stochastic matrices.
Property~\ref{def:substochastic_operator_continuation} of Definition~\ref{def:substochastic_operator} is a technical requirement to ensure the unique continuation from $L^p(\R_+)$ to $L^\infty(\R_+)$ for $1 \leq p < \infty$.
This property is often used in the study of (Sub-)Markovian operators and semigroups, see, for example, Section~1.6 in \cite{Fukushima.2010}, but unnecessary in the case of $2$-copulas and bounded domains.
A complete introduction to substochastic operators can be found in \cite{Bennett.1988}. 
In the following, we will establish a one-to-one correspondence between substochastic operators and subdistribution functions (see, Theorem~\ref{thm:isometry_subdistribution_substochastic}).
While many of the proofs work similarly to the case of compact spaces in \cite{Olsen.1996}, some care is needed due to the underlying non-finiteness of the measure space $\R_+$. 

\begin{definition}
A function $F: \R_+^d \rightarrow \R_+$ is called a subdistribution function if it is positive, $d$-increasing, bounded by $\TDLn^+$ and Lipschitz continuous with Lipschitz constant $1$. 
\end{definition}

Subdistribution functions constitute a generalization of tail dependence functions, where the condition of positive homogeneity is relaxed.

\begin{remark}
Note that the class of $d$-variate tail dependence functions equals the positive homogeneous subdistribution functions. 
\end{remark}

\begin{lemma} \label{lma:operator_to_subdistribution}
Let $T$ be a doubly substochastic operator. Then
\begin{equation*}
	F_T (x, y) := \cInt{T\indFunc{[0, y]}(s)}{0}{x}{s}
\end{equation*}
is a bivariate subdistribution function.
If $T$ is additionally equivariant, then $F_T$ is a bivariate tail dependence function, i.e.\ $F_T(\cdot) = \TDL{\cdot}{C}$ for some $C \in \Cd{2}$. 
\end{lemma}

\begin{proof}
We will check the properties \ref{prop:tail_dep_func_bounded}. - \ref{prop:tail_dep_func_homogeneous}. of Proposition~\ref{prop:tail_dep_func}.
\begin{enumerate}
	\item Because $0 \leq F_T$ is immediate for positive $T$, we only need to show that $F_T$ is bounded from above by $\TDL{\cdot}{C^+}$:
	\begin{equation*}
		\cInt{T\indFunc{[0, y]}(s)}{0}{x}{s} \leq \begin{cases}
			\cInt{T\indFunc{[0, y]}(s)}{0}{\infty}{s} \leq \cInt{\indFunc{[0, y]}(s)}{0}{\infty}{s} = y \\
			\cInt{T\indFunc{\R_+}(s)}{0}{x}{s} \leq x
		\end{cases} ~.
	\end{equation*}
	\item Let $R = [x_1, x_2] \times [y_1, y_2]$ with $x_1 \leq x_2$ and $y_1 \leq y_2$. Then the linearity of $T$ yields
	\begin{equation*}
		V_{F_T}(R) = \cInt{T \indFunc{[y_1, y_2]}}{x_1}{x_2}{s} \geq 0 ~.
	\end{equation*}
	\item Finally, we check the Lipschitz-continuity of $F_T$ with Lipschitz constant $1$. 
	For this, let $x_1, x_2, y_1$ and $y_2$ be in $\R_+$, then
	\begin{align*}
		\abs{F_T(x_2, y_2) - F_T(x_1, y_1)}	&\leq \abs{F_T(x_2, y_2) - F_T(x_1, y_2)} + \abs{F_T(x_1, y_2) - F_T(x_1, y_1)} \\
											&\leq \cInt{T \indFunc{[0, y_2]}(s)}{\min \braces{x_1, x_2}}{\max \braces{x_1, x_2}}{s} + \cInt{T \indFunc{[\min \braces{y_1, y_2}, \max \braces{y_1, y_2}]}(s)}{0}{x_1}{s} \\
											&\leq \norm{T \indFunc{[0, y_2]}}_\infty \abs{x_2 - x_1}  + \norm{T \indFunc{[\min \braces{y_1, y_2}, \max \braces{y_1, y_2}]}}_1 \\
											&\leq \norm{\indFunc{[0, y_2]}}_\infty \abs{x_2 - x_1} + \norm{\indFunc{[\min \braces{y_1, y_2}, \max \braces{y_1, y_2}]}}_1 \\
											&= 	\abs{x_2 - x_1}	+ \abs{y_2 - y_1} ~. 			
	\end{align*}
\end{enumerate}
Hence, $F_T$ is a bivariate subdistribution function. 
Finally, the positive homogeneity of $F_T$ follows from 
\begin{equation*}
	F_T(sx, sy) = \cInt{T\indFunc{[0, sy]}(t)}{0}{sx}{t} = \cInt{T \indFunc{[0, y]}\rbraces{\frac{t}{s}}}{0}{sx}{t} = \cInt{T \indFunc{[0, y]}\rbraces{z} s}{0}{x}{z} = s F_T(x, y)
\end{equation*}	
for any $s \geq 0$. 
Thus, $F_T$ is a positive homogeneous, bounded and $2$-increasing function, and the claim follows from Proposition~\ref{prop:tail_dep_func}. 
\end{proof}

\begin{lemma} \label{lma:subdistribution_to_operator}
Let $F$ be a bivariate subdistribution function. 
Then
\begin{align*}
	T_F: \quad	L^1(\R_+) + L^\infty(\R_+) &\rightarrow L^1(\R_+) + L^\infty(\R_+) \\
				T_F f(x) &= \partial_x \cInt{\partial_2 F(x, t) f(t)}{0}{\infty}{t}  
\end{align*}
defines a doubly substochastic operator.
Moreover, if $F$ is a bivariate tail dependence function, then $T_F$ is equivariant. 
\end{lemma}

\begin{proof}
As $x \mapsto \partial_2 F(x, t)$ is increasing, we have that for $\abs{f} \pm f \geq 0$
\begin{equation*}
	\cInt{\partial_2 F(x, t) \rbraces{\abs{f} \pm f}(t)}{0}{\infty}{t}
\end{equation*}
is again an increasing function in $x$ and its derivative with respect to the first component exists. 
Thus, representing $f$ as a linear combination of $\abs{f} + f$ and $\abs{f} - f$ implies that $T_F f$ exists.
Let us now verify properties \ref{def:substochastic_operator_positive}.-\ref{def:substochastic_operator_contraction}. of Definition~\ref{def:substochastic_operator}. 
\begin{enumerate}
	\item Let $f$ be positive. As $\partial_2 F(x_2, t)- \partial_2 F(x_1, t) \geq 0$ for $x_1 \leq x_2$, we have that
	\begin{equation*}
		\cInt{\partial_2 F(x_2, t) f(t)}{0}{\infty}{t} - \cInt{\partial_2 F(x_1, t) f(t)}{0}{\infty}{t} \geq 0
	\end{equation*}
	and hence $T_F f \geq 0$. 
	\item To prove \ref{def:substochastic_operator_mapping}. and \ref{def:substochastic_operator_contraction}., we first consider $f \in L^\infty(\R_+)$ and note that
	\begin{equation*}
		g(x) := \cInt{\partial_2 F(x, t) f(t)}{0}{\infty}{t} 
	\end{equation*}
	is Lipschitz continuous with Lipschitz constant $L = \norm{f}_\infty$.
	Because for $x_1 \leq x_2$, we have
	\begin{gather}
	\begin{aligned} \label{lma:subdistribution_to_operator_continuity}
	\abs{g(x_2) - g(x_1)} 	&= \abs{\cInt{\rbraces{\partial_2 F(x_2, t) - \partial_2 F(x_1, t)} f(t)}{0}{\infty}{t} }\\
							&\leq \norm{f}_\infty \cInt{\abs{\partial_2 F(x_2, t) - \partial_2 F(x_1, t)}}{0}{\infty}{t} \\
							&= \norm{f}_\infty \cInt{\rbraces{\partial_2 F(x_2, t) - \partial_2 F(x_1, t)}}{0}{\infty}{t} \\
							&= \norm{f}_\infty \lim\limits_{R \rightarrow \infty} \cbraces{F(x_2, t) - F(x_1, t)}_0^R 
							\leq \norm{f}_\infty \abs{x_2 - x_1} ~,
	\end{aligned}
	\end{gather}
	where the second equality is due to $x \mapsto \partial_2 F(x, t)$ being increasing, as $F$ is $2$-increasing.
	The last inequality follows from the Lipschitz-continuity of $F$.
	Thus, $T_F$ is non-expansive on $L^\infty(\R_+)$. \\
	Now let $f$ be in $L^1(\R_+)$. Combining the linearity and positivity of $T_F$ leads to 
	\begin{equation*}
		\abs{Tf} = \abs{T (f^+ - f^-)} \leq T f^+ + T f^- = T \abs{f} ~.
	\end{equation*}
	Thus, without loss of generality, let $f$ be positive. Then using the absolute continuity of $g$, we have
	\begin{align*}
		\cInt{T_F f (x)}{0}{\infty}{x}	&= \cInt{\partial_x \cInt{\partial_2 F(x, t) f(t)}{0}{\infty}{t}}{0}{\infty}{x} 
											= \lim\limits_{R \rightarrow \infty} \cInt{\partial_x \cInt{\partial_2 F(x, t) f(t)}{0}{\infty}{t}}{0}{R}{x} \\
											&= \lim\limits_{R \rightarrow \infty} \cInt{\partial_2 F(R, t) f(t)}{0}{\infty}{t}
											\leq \lim\limits_{R \rightarrow \infty} \cInt{f(t)}{0}{\infty}{t} 
											= \cInt{f(t)}{0}{\infty}{t}  ~,
	\end{align*}
	due to $0 \leq \partial_2 F(R, t) \leq 1$ and therefore $T_F$ is non-expansive on $L^1(\R_+)$. 
	\item It remains to show Property~\ref{def:substochastic_operator_continuation}. 
	To do so, we assume $f \geq 0$, otherwise one can use the decomposition $f = f^+ - f^-$ and treat $f^+$ and $f^-$ separately. 
	We first choose $f_n := f \indFunc{[0, n]} \nearrow f$ and set $h_n := f - f_n \searrow 0$.
	As $T_F h_n \geq 0$ is a decreasing sequence due to the positivity of $T_F$, it converges towards a measurable $h \geq 0$. 
	Moreover, 
	\begin{equation*}
		g_n(x)	:= \cInt{\partial_2 F(x, t) h_n(t)}{0}{\infty}{t}
				= \cInt{\partial_2 F(x, t) f(t)}{n}{\infty}{t}
				\rightarrow 0 ~,
	\end{equation*}
	as $\partial_2 F(x, t) \in L^1(\R_+) \cap L^\infty(\R_+)$ and $f \in L^\infty(\R_+)$. 
	Thus, for all $x \in \R_+$ we have by the absolute continuity of $g_n$
	\begin{align*}
		0 	&= \lim\limits_{n \rightarrow \infty} g_n(x)
			= \lim\limits_{n \rightarrow \infty} \cInt{g'_n(t)}{0}{x}{t}
			= \lim\limits_{n \rightarrow \infty} \cInt{T_F h_n (t)}{0}{x}{t} \\
			&= \cInt{\lim\limits_{n \rightarrow \infty} T_F h_n (t)}{0}{x}{t}
			= \cInt{h(t)}{0}{x}{t}
			\geq 0 ~,
	\end{align*}
	where the exchange of the integral and limit is legitimate due to an application of the monotone convergence theorem. 
	Therefore,  $h = 0$ holds almost everywhere and 
	\begin{equation*}
		T_F (f - f_n) = T_F h_n \searrow 0	\implies T_F f_n \nearrow T_F f ~.
	\end{equation*}
	Finally, suppose $f_n, g_k \in L^1(\R_+) \cap L^\infty(\R_+)$ and $f_n, g_k \nearrow f$. 
	Due to $\min\braces{f_n, g_k} \nearrow f_n$ as $k$ tends to $\infty$, it holds that
	\begin{equation*}
		T_F f_n = \sup\limits_{k \in \mathbb{N}} T_F(\min\braces{f_n, g_k}) 
				\leq \sup\limits_{k \in \mathbb{N}} T_F g_k ~.
	\end{equation*}
	Switching the roles of $f_n$ and $g_k$ yields the independence from the approximating sequence.
\end{enumerate}
Combining the previous three results, one sees that $T$ is an operator from $L^1(\R_+) + L^\infty(\R_+)$ onto $L^1(\R_+) + L^\infty(\R_+)$ and therefore doubly substochastic.
If $F$ is also positive homogeneous, then for any $s > 0$
\begin{align*}
	T_F	\indFunc{[0, sy]}(x)	&= \partial_x \cInt{\partial_2 F(x, t) \indFunc{[0, sy]}(t)}{0}{\infty}{t} 
								= \partial_x \cInt{\partial_2 F(x, t) \indFunc{[0, y]}\rbraces{\frac{t}{s}}}{0}{\infty}{t} \\
								&= \partial_x \cInt{\partial_2 F(x, sz) \indFunc{[0, y]}\rbraces{z} s}{0}{\infty}{z} \\
								&= \partial_x \cInt{\partial_2 F \rbraces{\frac{x}{s}, z} \indFunc{[0, y]}\rbraces{z}}{0}{\infty}{z} 
								= T_F \indFunc{[0, y]}\rbraces{\frac{x}{s}} ~. \qedhere
\end{align*}
\end{proof}

As a consequence of Lemma~\ref{lma:operator_to_subdistribution} and \ref{lma:subdistribution_to_operator}, we obtain the main result establishing the correspondence between subdistribution functions and substochastic operators. 

\begin{theorem} \label{thm:isometry_subdistribution_substochastic}
Let $F$ be a bivariate subdistribution function and $T$ a substochastic operator, and define $\Phi(T) := F_T$ and $\Psi(F) := T_F$.
Then $\Phi \circ \Psi$ and $\Psi \circ \Phi$ define identities on their respective spaces. 
Furthermore, $F$ is positive homogeneous if and only if $T_F$ is equivariant. 
\end{theorem}

\begin{proof}
First, let $F$ be a subdistribution function. Then using the Lipschitz continuity of $F$
\begin{align*}
	\Phi \circ \Psi (F) (x, y)	&= \cInt{\Psi(F) \indFunc{[0, y]}(s)}{0}{x}{s}									
								= \cInt{\partial_s \cInt{\partial_2 F(s, t) \indFunc{[0, y]}(t)}{0}{\infty}{t}}{0}{x}{s} \\
								&= \cInt{\partial_s \cInt{\partial_2 F(s, t)}{0}{y}{t}}{0}{x}{s}
								= \cInt{\partial_s F(s, y)}{0}{x}{s}
								= F(x, y) ~.
\end{align*}
Conversely, let $T$ be a substochastic operator and $f(t) = \indFunc{[0, y]}(t)$. 
Then the absolute continuity yields
\begin{align*}
	\Psi \circ \Phi (T) f(x)	&= \partial_x \cInt{\partial_2 \Phi(T)(x, t) f(t)}{0}{\infty}{t}
								= \partial_x \cInt{\partial_t \cInt{T \indFunc{[0, t]}(s)}{0}{x}{s} f(t)}{0}{\infty}{t} \\
								&= \partial_x \cInt{\partial_t \cInt{T \indFunc{[0, t]}(s)}{0}{x}{s}}{0}{y}{t} 
								= \partial_x \cInt{T \indFunc{[0, y]}(s)}{0}{x}{s}
								= T \indFunc{[0, y]}(x) ~.
\end{align*}
Thus $\Psi \circ \Phi (T)$ and $T$ are substochastic operators which agree on $[0, y]$ and, following the argument in Lemma 2.2 of \cite{Olsen.1996}, agree on $L^1(\R_+)$. 
Property~\ref{def:substochastic_operator_continuation} of Definition~\ref{def:substochastic_operator} now ensures that $T$ and $\Psi \circ \Phi$ conincide on $L^1(\R_+) + L^\infty(\R_+)$.
Finally, Lemma~\ref{lma:operator_to_subdistribution} and \ref{lma:subdistribution_to_operator} yield the equivalence between the positive homogeneity of $F$ and the equivariance of $T$. 
\end{proof}
The correspondence between substochastic operators and subdistribution functions is a structure-preserving isomorphism translating $*$ into $\circ$ and vice versa.  
To verify the structure preservation property, we need to introduce a slight generalization of the previously introduced Markov product for tail dependence functions. 
The Markov product of two subdistribution functions $F$ and $G$ is defined for all $\boldm{w} \in \R_+^2$ as
\begin{equation*}
	(F * G) (\boldm{w})	= \cInt{\partial_2 F(w_1, t) \partial_1 G(t, w_2)}{0}{\infty}{t} ~.
\end{equation*}
We want to show that $F*G$ is again a subdistribution function.
Applying Remark~\ref{remark:markov_product_subdistribution}, it only remains to show the Lipschitz continuity of $(F * G)$, which follows as in Equation~\eqref{lma:subdistribution_to_operator_continuity}.

\begin{theorem}
Let $F$ and $G$ be subdistribution functions. 
Then
\begin{equation*}
	T_{F * G} = T_{F} \circ T_{G} ~.
\end{equation*}
\end{theorem}

\begin{proof}
In view of Theorem~\ref{thm:isometry_subdistribution_substochastic}, it suffices to prove that
\begin{equation*}
	\Phi \rbraces{T_F \circ T_G} (\boldm{w}) = \Phi \rbraces{T_{F * G}}(\boldm{w}) = (F * G)(\boldm{w})
\end{equation*}
for all $w \in \R_+^2$. 
To do so, we use the Lipschitz continuity to obtain
\begin{align*}
	\Phi \rbraces{T_F \circ T_G} (\boldm{w})	&= \cInt{(T_F \circ T_G) \indFunc{[0, w_2]}(s)}{0}{w_1}{s}
												= \cInt{\partial_s \cInt{\partial_2 F(s, t) T_G \indFunc{[0, w_2]}(t)}{0}{\infty}{t}}{0}{w_1}{s} \\
												&= \cInt{\partial_2 F(w_1, t) T_G \indFunc{[0, w_2]}(t)}{0}{\infty}{t} \\
												&= \cInt{\partial_2 F(w_1, t)  \partial_t \cInt{\partial_2 G(t, s) \indFunc{[0, w_2]}(s)}{0}{\infty}{s}}{0}{\infty}{t} \\
												&= \cInt{\partial_2 F(w_1, t)  \partial_1 G(t, w_2)}{0}{\infty}{t} 
												= F * G(\boldm{w}) ~. \qedhere
\end{align*}
\end{proof}

The Banach space adjoint of a substochastic operator $T_F$ corresponds to the doubly substochastic operator associated with the transpose $F^T$ of $F$ where $F^T(x, y) := F(y, x)$. 

\begin{proposition}
Let $f \in L^1(\R_+) + L^\infty (\R_+)$ and $g \in L^1(\R_+) \cap L^\infty (\R_+)$. 
Then
\begin{equation*}
	\cInt{(T_F f)(x) g(x)}{0}{\infty}{x} = \cInt{f(x) (T_{F^T} g)(x)}{0}{\infty}{x} ~.
\end{equation*}
\end{proposition}

\begin{proof}
Let $f \in L^1(\R_+) + L^\infty (\R_+)$ and $g \in L^1(\R_+) \cap L^\infty (\R_+)$.
As the space of compactly supported and smooth functions is dense in $L^1(\R_+) \cap L^\infty (\R_+)$, we only need to show the desired result for $g \in \mathcal{C}^\infty_0(\R_+)$. 
Then an identical calculation to Lemma 2.4 from \cite{Olsen.1996} yields the result, except that for the partial integration, we additionally require $g(\infty) = 0$ and $\partial_2 F(0, t) = 0$, which holds due to $F(0, t) \equiv 0$. 
\end{proof}

Using this connection between the adjoint of $T$ and the transpose of $F$, we can establish a relation between strict subdistribution functions and Markov operators. 

\begin{definition}
Let $F$ be a bivariate subdistribution function. 
Then we call $F$ strict if 
\begin{equation*}
	\lim\limits_{t \rightarrow \infty} F(w_1, t) = w_1 \; \text{ and } \; \lim\limits_{t \rightarrow \infty} F(t, w_2) = w_2
\end{equation*}
for all $(w_1, w_2)$ in $\R_+^2$. 
\end{definition}

\begin{definition}
Let $T$ be a doubly substochastic operator. 
$T$ is called a doubly stochastic operator or Markov operator if
\begin{equation*}
	T \indFunc{\R_+} = \indFunc{\R_+} \; \text{ and } \cInt{Tf (x)}{0}{\infty}{x} = \cInt{f(x)}{0}{\infty}{x}
\end{equation*}
for all $f$ in $L^1(\R_+)$. 
\end{definition}

\begin{proposition}
Let $F$ be a bivariate subdistribution function. Then $F$ is strict if and only if $T_F$ and $T_{F^T}$ are Markov operators. 
\end{proposition}

\begin{proof}
First, let $F$ be strict. 
Then,
\begin{align*}
	T_F \indFunc{\R_+} (x) 	&= \partial_x \cInt{\partial_2 F(x, t) \indFunc{\R_+}(t)}{0}{\infty}{t}
							= \partial_x \cInt{\partial_2 F(x, t)}{0}{\infty}{t} \\
							&= \partial_x \rbraces{F(x, \infty) - F(x, 0)}
							= \partial_x x
							= \indFunc{\R_+}(x) 
\end{align*}
for all $x \in \R_+$.
Now let $f$ be in $L^1(\R_+)$, then it holds
\begin{align*}
	\cInt{T_F f(x)}{0}{\infty}{x}	&= \cInt{\partial_x \cInt{ \partial_2 F(x, t) f(t)}{0}{\infty}{t}}{0}{\infty}{x} 
									= \cInt{ \partial_2 F(\infty, t) f(t)}{0}{\infty}{t}
									= \cInt{f(t)}{0}{\infty}{t}
\end{align*}
using the strictness of $F$ and absolute continuity. 	
The claims for $T_{F^T}$ can be proven analogously by exploiting the strictness in the second component of $F$. 
Conversely, if $T_F$ and $T_{F^T}$ are doubly stochastic, then
\begin{align*}
	\lim\limits_{t \rightarrow \infty} F(t, w_2)	= \lim\limits_{t \rightarrow \infty} \cInt{T_F \indFunc{[0, w_2]}(s)}{0}{t}{s}
													= \cInt{T_F \indFunc{[0, w_2]}(s)}{0}{\infty}{s}
													= \cInt{\indFunc{[0, w_2](s)}}{0}{\infty}{s}
													= w_2
\end{align*}
and, analogously, for $\lim\limits_{t \rightarrow \infty} F(w_1, t) = w_1$. 
\end{proof}

Finally, we present an alternative proof of Theorem~\ref{thm:dependence_reduction}, using the theory of substochastic operators.

\begin{proof}[Proof of Theorem~\ref{thm:dependence_reduction}]
For every substochastic operator $T$ and every $t \in [0, \infty)$, it holds
\begin{equation*}
	\cInt{(Tf)^*(s)}{0}{t}{s} \leq \cInt{f^*(s)}{0}{t}{s}
\end{equation*}
or, in short, $Tf \prec f$, where $f^*$ denotes the decreasing rearrangement of $f$ (see, Chapter 1 in \cite{Bennett.1988}). 
Thus
\begin{align*}
	\partial_1 \Lambda_2 (w_1, w_2)	&\succ T_{\Lambda^T_1} \partial_1 \Lambda_2 (\cdot, w_2) (w_1)	\\
									&= \partial_1 \cInt{\partial_2 \Lambda_1(w_1, s) \partial_1 \Lambda_2(s, w_2)}{0}{\infty}{s} \\
									&= \partial_1 (\Lambda_1 * \Lambda_2) (w_1, w_2)
\end{align*}
together with the monotonicity of the tail dependence function yields
\begin{equation*}
	(\Lambda_1 * \Lambda_2) (w_1, w_2) \leq \Lambda_2 (w_1, w_2) ~. \qedhere
\end{equation*}
\end{proof}

\section*{Acknowledgements}
We are grateful to Piotr Jaworski for interesting discussions. 
We thank an anonymous referee for their constructive comments, which helped us to improve the article.
The second author gratefully acknowledges financial support from the German Academic Scholarship Foundation.

\bibliography{reference}

\end{document}